\documentclass[11pt]{article}
\usepackage{latexsym,color,amsmath,amsthm,amssymb,amscd,amsfonts,latexsym}
\usepackage[all]{xy}

\newtheorem{lemma}{Lemma}[section]
\newtheorem{proposition}[lemma]{Proposition}
\newtheorem{theorem}[lemma]{Theorem}
\newtheorem*{theorem*}{Theorem}

\theoremstyle{definition}
\newtheorem{claim}[lemma]{Claim}
\newtheorem{corollary}[lemma]{Corollary}
\newtheorem{definition}[lemma]{Definition}

\theoremstyle{remark}

\newtheorem*{notation*}{\bf Notation}
\newtheorem{remark}[lemma]{\bf Remark}
\newtheorem*{convention*}{\bf Convention}
\newtheorem*{question*}{Question}

\theoremstyle{plain}
\newtheorem{lemin}{Lemma}

\newtheorem{propin}[lemin]{Proposition}
\newtheorem{thmin}[lemin]{Theorem}

\theoremstyle{definition}

\newtheorem{corolin}[lemin]{Corollary}

\theoremstyle{remark}

\newcommand{\Hom}{{\rm Hom}}

\newcommand{\Ker}{{\rm Ker}}
\newcommand{\rank}{{\rm rank}}

\newcommand{\Span}{{\rm Span}}
\newcommand{\Spec}{{\rm Spec}}

\newcommand{\CCC}{{\mathbb C}}
\newcommand{\FF}{{\mathbb F}}

\newcommand{\RR}{{\mathbb R}}
\newcommand{\ZZ}{{\mathbb Z}}
\newcommand{\PP}{{\mathbb P}}
\newcommand{\NN}{{\mathbb N}}

\newcommand{\QQ}{{\mathbb Q}}
\newcommand{\KK}{{\mathbb K}}
\newcommand{\LL}{{\mathbb L}}

\newcommand{\fm}{{\mathfrak m}}
\newcommand{\fn}{{\mathfrak n}}

\newcommand{\CL}{{\mathcal L}}

\newcommand{\CO}{{\mathcal O}}

\newcommand{\CN}{{\mathcal N}}

\newcommand{\CJ}{{\mathcal J}}

\newcommand{\CK}{{\mathcal K}}

\newcommand{\Id}{{{\mathchoice {\rm 1\mskip-4mu l} {\rm 1\mskip-4mu l}
      {\rm 1\mskip-4.5mu l} {\rm 1\mskip-5mu l}}}}

\makeatletter \@addtoreset {equation}{section}

\renewcommand\theequation
  {\ifnum \c@subsection>\z@ \arabic{section}.\arabic{subsection}.\arabic{equation}
  \else \arabic{section}.\arabic{equation} \fi}
\makeatother

\begin{document}

\author{Yaron Ostrover\footnote{The first named author was supported by NSF
grant DMS-0706976.} \ and Ilya Tyomkin}
\title{On the quantum homology algebra of toric Fano manifolds} 
\maketitle

\begin{abstract}
 \noindent In this paper we study certain algebraic properties of
the quantum homology algebra for the class of symplectic toric Fano
manifolds. In particular, we examine the semi-simplicity of the
quantum homology algebra, and the more general property of
containing a field as a direct summand. Our main result provides an
easily-verifiable sufficient condition for these properties which is
independent of the symplectic form. Moreover, we answer two
questions of Entov and Polterovich negatively by providing examples
of toric Fano manifolds with non semisimple quantum homology
algebra, and others in which the Calabi quasimorphism in
non-unique.
\end{abstract}

\section{Introduction.} \label{sec-intr}

The quantum homology algebra $QH_*(X,\omega)$ of a symplectic
manifold $(X,\omega)$ is, roughly speaking, the singular homology
of $X$ endowed with a modified algebraic structure, which is a
deformation of the ordinary intersection product. It was
originally introduced by the string theorists Vafa and
Witten~\cite{V1},\cite{W} in the context of topological quantum
field theory, followed by a rigorous mathematical construction by
Ruan and Tian~\cite{RT2} in the symplectic setting, and by
Kontsevich and Manin~\cite{KM} in the algebra-geometric setting.

Since its introduction in $1991$, there has been a great deal of
interest in the study of quantum homology from various
disciplines, both by physicists and mathematicians. In particular,
the quantum homology algebra plays an important role in symplectic
geometry where, for example, it is ring-isomorphic to the Floer
homology. Recently, the study of quantum homology had a profound
impact in the realm of algebraic geometry, where ideas from string
theory have led to astonishing predictions regarding enumerative
geometry of rational curves. Furthermore, quantum homology
naturally arises in string theory, where it is an essential
ingredient in the A-model side of the mirror symmetry phenomenon.
We refer the reader to~\cite{MS} and~\cite{HKLPTVVZ} and the
references within for detailed expositions to the theory of
quantum homology.

In this paper we focus on the following algebraic properties of
the quantum homology algebra. Recall that a finite dimensional
commutative algebra over a field is said to be {\it semisimple}
if it decomposes into a direct sum of fields. A more general
property is the following: a finite dimensional commutative
algebra $A$ is said to contain a field as a direct summand if it
splits (as an algebra) into a direct sum $A=A_1 \oplus A_2$, where
$A_1$ is a field and no assumptions on the algebra $A_2$ are
imposed. We wish to remark that there are several different
notions of semi-simplicity in the context of quantum homology (see
e.g.~\cite{Do},~\cite{KM}). The semi-simplicity we consider here
was first examined by Abrams~\cite{A}.

Our main motivation to study the above mentioned algebraic
properties of the quantum homology algebra is the recent works by
Entov and Polterovich on Calabi quasimorphisms and symplectic
quasi-states (\cite{EP1},~\cite{EP2},~\cite{EP3},~\cite{EP4}), in
which the algebraic structure of the quantum homology plays a key
role. More precisely, our prime object of interest is the
subalgebra $QH_{2d}(X,\omega)$, i.e. the graded part of degree
$2d = {\rm dim_{\RR}} \, X$ of the quantum homology
$QH_*(X,\omega)$. This subalgebra is finite dimensional over a
field  $\KK^{\scriptscriptstyle \downarrow}$ (see Subsection~\ref{qh-sym-def} for the
definitions). In what follows, we say that $QH_{2d}(X,\omega)$ is
semisimple if it is semisimple as a $\KK^{\scriptscriptstyle
\downarrow}$-algebra.

The following theorem has been originally proven in the case of
monotone symplectic manifolds in~\cite{EP1} (using a slightly
different setting), then generalized by the first named author
in~\cite{O} to the class of rational strongly semi-positive
symplectic manifolds that satisfy some technical condition which
was eventually removed in~\cite{EP2}.
\begin{theorem*} \label{EPO-thm}
Let $(X,\omega)$ be a rational\footnote{It is very plausible that
the rationality assumption can be removed due to the recent works of
Oh~\cite{Oh}, and
Usher~\cite{Us}.} 
strongly semi-positive symplectic manifold of
dimension $2d$ such that the quantum homology subalgebra
$QH_{2d}(X,\omega) \subset QH_*(X,\omega)$ is semisimple. Then $X$ admits a Calabi quasimorphism and a symplectic
quasi-state.
\end{theorem*}

For the definition of Calabi quasimorphisms and symplectic
quasi-states, and detailed discussion of their application in
symplectic geometry we refer the reader to~\cite{EP1},~\cite{EP4}.
We wish to mention here that other than demonstrating applications
to Hofer's geometry and $C^0$-symplectic topology, Entov and
Polterovich used the above theorem to obtain Lagrangian
intersection type results. For example, in~\cite{BEP} they proved
(together with Biran) that the Clifford torus in ${\mathbb C}P^n$
is not displaceable by a Hamiltonian isotopy. In a later
work~\cite{EP3}, they proved the non-displaceability of certain
singular Lagrangian submanifolds, a result which is currently out
of reach for the conventional Lagrangian Floer homology technique.
We refer the reader to~\cite{EP3} for more details in this
direction.

Very recently, McDuff pointed out that the semi-simplicity
assumption in the above theorem can be relaxed to the weaker
assumption that $QH_{2d}(X,\omega)$ contains a field as a direct
summand. Moreover, she showed that in contrast with
semi-simplicity, this condition holds true for one point blow-ups
of non-uniruled symplectic manifolds such as the standard
symplectic four torus $T^4$ (see~\cite{Mc} and~\cite{EP2} for
details), consequently enlarging the class of manifolds admitting
Calabi quasimorphisms and symplectic quasi-states. Thus, in what
follows we will study not only the semi-simplicity of the quantum
homology algebra, but also the more general property of containing
a field as a direct summand.

A different motivation to study the semi-simplicity of the quantum
homology algebra is due to a work of Biran and Cornea.
In~\cite{BC} they showed that in certain cases the semi-simplicity
of the quantum homology implies restrictions on the existence of
certain Lagrangian submanifolds. We refer the reader to~\cite{BC},
Subsection 6.5 for more details.

Finally, a third motivation comes from physics, where in the
symplectic toric Fano case the semi-simplicity of the quantum
homology algebra implies that the corresponding $\CN = 2$
Landau-Ginzburg model is massive. The physical interpretation is
that the theory has massive vacua and the infrared limit of this
model is trivial. See~\cite{HKLPTVVZ} and the references within for
precise definition and discussion.

Examples of symplectic manifolds with semisimple quantum homology
are ${\mathbb C}P^d$ (see e.g.~\cite{EP1}); 
complex Grassmannians; and the smooth complex quadric $Q = \{ z_0^2 + \cdots + z_d^2 - z_{d+1}^2
=0\} \subset {\mathbb C}P^{d+1}$ (see~\cite{A} for the last two
examples). As mentioned above, McDuff (see~\cite{Mc}
and~\cite{EP2}) provides a large class of examples of symplectic
manifolds whose quantum homology contains a field as a direct
summand but is not semisimple, by considering the one point
blow-up of a non-uniruled symplectic manifold. Using the
K\"{u}nneth formula for quantum homology, one can show that both
semi-simplicity and the property of containing a field as a direct
summand are preserved when taking products (see~\cite{EP2}).

Another class of examples are toric Fano $2$-folds. Recall that up
to rescaling the symplectic form by a constant factor there are
exactly five symplectic toric Fano $2$-folds: ${\mathbb C}P^1
\times {\mathbb C}P^1$, ${\mathbb C}P^2$, and the blowups of
${\mathbb C}P^2$ at $1,2$ and $3$ points. The following theorem is
a combination of results from~\cite{O} and~\cite{EP2}.
\begin{theorem*} If $(X,\omega)$ is a symplectic toric Fano $2$-fold  then $QH_4(X,\omega)$ is semisimple.
\end{theorem*}
In view of the above, Entov and Polterovich posed the following
question in~\cite{EP2}:

\noindent {\bf Question:} Is it true that the algebra
$QH_{2d}(X,\omega)$ is semisimple for any symplectic toric Fano
manifold $(X,\omega)$?

It is known (see e.g.~\cite{I} Corollary 5.12, and~\cite{FOOO}
Proposition 7.6) that semi-simplicity holds for generic toric
symplectic form. For the sake of completeness, we include this
statement below. More precisely:

\begin{thmin} \label{thmin-generic} Let $X$ be a smooth $2d$-dimensional toric Fano variety.
Then for a generic\footnote{The space of toric symplectic forms
has natural structure of a topological space, and generic here
means that $\omega$ belongs to a certain open dense subset in this
space.} choice of a toric symplectic form $\omega$ on $X$, the
quantum homology $QH_{2d}(X, \omega)$ is semisimple.
\end{thmin}

However, it turns out that the answer to the question of Entov and
Polterovich is negative. The first counter example exists in (real)
dimension eight.

{\begin{propin} \label{counter-ex} There exists a
monotone\footnote{Recall that $(X,\omega)$ is called monotone if
$c_1=\kappa [\omega]$, where $\kappa
>0$, and $c_1$ is the first Chern class of $X$.} 
symplectic toric Fano 4-fold $(X,\omega)$ whose quantum homology
algebra $QH_{8}(X,\omega)$ is not semisimple.
\end{propin}

Using K\"unneth formula we also produce examples of non-monotone
symplectic Fano manifolds $(X,\omega)$ with non semisimple
quantum cohomology algebras. In particular, there exists a
non-monotone Fano 5-fold $(X,\omega)$ with a non semisimple
$QH_{10}(X,\omega)$. Notice  that it would be interesting to
construct an example of non-decomposable non-monotone symplectic
Fano manifold with this property.

We wish to remark that a toric Fano manifolds $X$ may be equipped
with a distinguished toric symplectic form $\omega_0$, namely the
normalized monotone symplectic form corresponding to $c_1(X)$. This is
the unique symplectic form for which the corresponding moment
polytope is reflexive (see Section~\ref{Prelim-section}).
 Our second result shows that as far as
semi-simplicity is concerned, the symplectic form $\omega_0$ is,
in a matter of speech, the worst. 

\begin{thmin}\label{thmin:sms}
Let $X$ be a toric Fano manifold of (real) dimension $2d$, and let
$\omega$ be a toric symplectic form on $X$. If
$QH_{2d}(X,\omega_0)$ is semisimple then $QH_{2d}(X,\omega)$ is
semisimple.
\end{thmin}

Inspired by McDuff's observation we modify the above question of
Entov and Polterovich and ask the following:

\noindent {\bf Question:} 
Is it true that the algebra $QH_{2d}(X,\omega)$ contains a field
as a direct summand for any symplectic toric Fano manifold
$(X,\omega)$?

Currently we do not have an example of a symplectic toric Fano
manifold $(X,\omega)$ that does not satisfy this property.
Moreover, it seems that no such example exists in low dimensions.
We hope to return to this question in the near future. Meanwhile,
we prove the following analog of Theorem \ref{thmin:sms}:

\begin{thmin}\label{thmin:dirsmd} Let $X$ be a toric Fano manifold of (real) dimension $2d$, and let
$\omega$ be a toric symplectic form on $X$. If
$QH_{2d}(X,\omega_0)$ contains a field as a direct summand, then
$QH_{2d}(X,\omega)$ contains a field as a direct summand.
\end{thmin}
In Subsection~\ref{sec:LGP}  we show that the property of
$QH_{2d}(X,\omega)$ of having a field as a direct summand is
equivalent to the existence of a non-degenerate critical point of a
certain (combinatorially defined) function $W_X$, called the
Landau-Ginzburg superpotential, assigned naturally to $(X, \omega)$.
McDuff's observation and  Theorem~\ref{thmin:dirsmd} reduce the
question of the existence of  Calabi quasimorphisms and symplectic
quasi-states on a symplectic toric manifold $(X,\omega)$ to the
normalized monotone case $(X, \omega_0)$, and hence to the problem
of analyzing the critical points of a function $W_X$, depending only
on $X$ and not on the symplectic form. This can be done easily in
many cases. In particular we construct the following new examples of
symplectic manifolds admitting Calabi quasimorphisms and symplectic
quasi-states:

\begin{corolin} \label{cor:examlpes}
Let $X$ be one of the following manifolds: (i) a symplectic toric
Fano 3-fold, (ii) a symplectic toric Fano 4-fold, (iii) the
symplectic blow up of ${\mathbb C}P^d$ at $d+1$ general points.
Then $X$ admits a Calabi quasimorphism and a symplectic
quasi-state.
\end{corolin}

Another byproduct of our method is the following two propositions.
The first one, inspired by McDuff~\cite{Mc1}, answers a question
raised by Entov and Polterovich~\cite{EP1} regarding the uniqueness
of the Calabi quasimorphism. We will briefly recall the definition
of a Calabi quasimorphism in Section~\ref{sec:non-unique}. For a
detailed discussion see~\cite{EP1},~\cite{EP4}.

\begin{corolin} \label{cor:quasi-non-unique}
Let $(X,\omega)$ be the blow up of $\CCC P^2$ at one point equipped
with a symplectic form $\omega$. If
$\omega(L)/\omega(E)<3$, where $L$ is the class of a line on $\CCC P^2$, and $E$ is the class of the exceptional divisor, then
there are 
two different Calabi quasimorphisms on
$(X,\omega)$.
\end{corolin}
\noindent {\bf Remark:} Other examples of symplectic manifolds for
which the Calabi quasimorphism is non-unique were constructed by
Entov, McDuff, and Polterovich in~\cite{BEP}. We chose to include
the above example here due to the simplicity of the argument.
Moreover, we remark that Corollary~\ref{cor:quasi-non-unique} can be
easily extended to other toric Fano manifolds.

Finally, we finish this section with a folklore result, known to
experts in the field and proven in full detail by Auroux (see
Theorem $6.1$ in~\cite{Au}). We wish to remark that the results
in~\cite{Au} are more general (see Proposition $6.8$ in~\cite{Au}),
and do not rely on Batyrev's description of the quantum homology
algebra. 
However, since by using
Proposition~$(\ref{prop:LGmodel})$ the proof of the claim below
becomes much simpler, we felt it might be useful to include it here
as well.
\begin{corolin} \label{cor:mult-by-c1}
For a smooth toric Fano manifold $X$, the critical values of the
superpotential $W_X$ are the eigenvalues of the linear operator $QH^0(X,\omega)
\rightarrow QH^0(X,\omega)$ given by multiplication by $q^{-1}c_1(X)$.
\end{corolin}

\noindent {\bf Structure of the paper:} In
Section~\ref{Prelim-section} we recall some basic definitions and
notations regarding symplectic toric manifolds. In
Section~\ref{subsec-qh} we give three equivalent description of the
quantum cohomology of toric Fano manifolds. In
section~\ref{prf-main-results} we prove our main results. For
technical reasons it is more convenient for us to use quantum
cohomology instead of homology. In this setting
Theorem~\ref{thmin-generic} becomes Theorem~\ref{thm-generic} and
Theorems~\ref{thmin:sms} and~\ref{thmin:dirsmd} are combined
together to Theorem~\ref{main-thm}. In Section~\ref{sec:examples} we
prove Proposition~\ref{counter-ex} and Corollary~\ref{cor:examlpes}.
In Sections~\ref{sec:non-unique} and~\ref{sec:mult-by-c1} we prove
Corollaries~\ref{cor:quasi-non-unique} and~\ref{cor:mult-by-c1}
respectively. Finally, in the Appendix we give a short review on
toric varieties.

\noindent{\bf Acknowledgement:} We thank D. Auroux, L. Polterovich,
P. Seidel, and M. Temkin for helpful comments and discussions.

\section{Preliminaries, notation, and conventions.} \label{Prelim-section}

In this section we recall some algebraic definitions 
and collect all the facts we need regarding  symplectic toric
manifolds.

\subsection{Algebraic preliminaries}

\begin{convention*}
All the rings and algebras in this paper are 
commutative with unit element.
\end{convention*}

\subsubsection{Semigroup algebras.}
Let $G$ be a commutative semigroup and let $R$ be a ring. 
The semigroup algebra $R[G]$ is the $R$-algebra consisting of
finite sums of formal monomials $x^g$, $g \in G$, with
coefficients in $R$, and equipped with the natural algebra
operations. For example, if $G=\ZZ^d$ then $R[G]$ is the algebra
of Laurent polynomials $R[x_1^{\pm 1},\dotsc ,x_d^{\pm 1}]$, and if
$G=\ZZ_+^d$ then $R[G]$ is the polynomial algebra
$R[x_1,\dotsc ,x_d]$. In this paper $R$ will usually be either the
field $\KK$ or the Novikov ring $\Lambda$ which are introduced at
the end of Subsection~\ref{qh-sym-def}.

\subsubsection{Semisimple algebras.} Among the many equivalent definitions of semisimplicity  we
consider the following:
\begin{definition} \label{def-semisimple}
Let $\FF$ be a field. A finite dimensional $\FF$-algebra $A$ is
called {\it semisimple} if  it contains no nilpotent elements.
\end{definition}
In the language of algebraic geometry (see e.g.~\cite{EH}),
semisimplicity is equivalent to the affine scheme $\Spec A$ being
reduced and finite over $\Spec\, \FF$, and in particular
zero-dimensional. Notice that a Noetherian zero-dimensional scheme is reduced
if and only if it is regular. If in addition $char\,\FF=0$ this is
equivalent to $\Spec A$ being geometrically regular (i.e., $\Spec
A\otimes_\FF\overline{\FF}$ is smooth). It follows from this
geometric description that a finite dimensional algebra $A$ is
semisimple if and only if it is a direct sum of field extensions
of $\FF$. Moreover, if $char\,\FF=0$ then $A$ is semisimple if
and only if $A\otimes_\FF\LL$ is semisimple for {\it any} field
extension $\LL/\FF$.

We say that $\FF$-algebra $A$ {\it contains a field as a direct summand} if it
decomposes {\it as a $\FF$-algebra} into a direct sum $A=\LL \oplus A'$,
where $\LL/\FF$ is a field extension.
Again, in geometric terms this condition means that the affine
scheme $\Spec A$ contains a regular point as an irreducible
component.

\subsubsection{Non-Archimedean seminorms.}\label{subsec:nonArch}
Let $\FF$ be a field. A {\it non-Archimedean norm} is a function
$|\cdot|\colon\FF\to \RR_+$ satisfying the following properties:
$|\lambda\mu|=|\lambda||\mu|$, $|\lambda+\mu|\le
\max\{|\lambda|,|\mu|\}$, and $|\lambda|=0$ if and only if $\lambda=0$.
Notice that the norm $|\cdot|$ defines a metric on $\FF$. A field
$\FF$ is called {\it non-Archimedean} if it is equipped with a
non-Archimedean norm such that $\FF$ is complete (as a metric
space). One can define the corresponding {\it non-Archimedean
valuation $\nu \colon \FF\rightarrow \RR \cup \{-\infty\}$} on $\FF$ by
setting\footnote{Usually one defines $\nu(\lambda):=-\log |\lambda|$
and $\nu(0)=\infty$, however, we chose the above 
normalization to make it compatible with~\cite{EP3} and~\cite{MT}}
$\nu(\lambda):=\log |\lambda|$. It satisfies similar properties,
i.e. $\nu(\lambda\mu)=\nu(\lambda)+\nu(\mu)$, $\nu(\lambda+\mu)\le
\max\{\nu(\lambda),\nu(\mu)\}$, and $\nu(\lambda)=-\infty$ if and
only if $\lambda=0$.

Let $\FF$ be a non-Archimedean field, and let $A$ be an $\FF$-algebra. A {\it non-Archimedean seminorm} on
$A$ is a function $\|\cdot\|\colon A\to \RR_+$ such that $\|fg\|\le\|f\|\|g\|$, $\|f+g\|\le \max\{\|f\|,\|g\|\}$,
and $\|\lambda f\|=|\lambda|\|f\|$ for all $\lambda\in\FF$, $f,g\in A$. A seminorm is called {\it norm} if
the following holds: $\|f\|=0$ if and only if $f=0$. It is well known that if $\|\cdot\|$ is a non-Archimedean
seminorm and $\|f\|\ne\|g\|$ then $\|f+g\|=\max\{\|f\|,\|g\|\}$.
Given a non-Archimedean seminorm $\|\cdot\|$ one can consider the associated {\it spectral seminorm}
$\|\cdot\|_{\rm sp}$ defined by $\|f\|_{\rm sp}=\lim_{k\to\infty}\sqrt[k]{\|f^k\|}$.
It is easy to check that $\|\cdot\|_{\rm sp}$ is a non-Archimedean seminorm on $A$ satisfying $\|f^k\|_{\rm sp}=\|f\|_{\rm sp}^k$ for all $k$. Notice however, that $\|\cdot\|_{\rm sp}$ need not be a norm even if $\|\cdot\|$ is.

\begin{lemma} \label{lem-about-non-arc-norms}
Let $(\FF,|\cdot|)$ be a non-Archimedean algebraically closed field,
and let $A$ be a finite $\FF$-algebra equipped with a
non-Archimedean norm $\|\cdot\|$. Let $B\subseteq A$ be a local
$\FF$-subalgebra, $\fm$ its maximal ideal, and $e_B\in B$ its unit
element. Then $B=\FF e_B\oplus\fm$ as $\FF$-modules, and $\|\lambda
e_B+g\|_{\rm sp}=|\lambda|$ for all $\lambda\in\FF$ and $g\in\fm$.
\end{lemma}

\begin{proof} The field $B/\fm$ is a finite extension of $\FF$, thus $B/\fm=\FF$ since $\FF$ is algebraically closed; the decomposition now follows. Notice that $B$ is finite over $\FF$ thus any element $g\in\fm$ is nilpotent, hence $\|g\|_{\rm sp}=0$. Notice that $\|e_B\|\ne 0$ since $\|\cdot\|$ is a norm, hence $\|e_B\|_{\rm sp}=\lim_{k\to\infty}\sqrt[k]{\|e_B^k\|}=\lim_{k\to\infty}\sqrt[k]{\|e_B\|}=1$. Thus $\|\lambda e_B\|_{\rm sp}=|\lambda|>0=\|g\|_{\rm sp}$ for any $0\ne \lambda\in\FF$ and $g\in\fm$, which implies $\|\lambda e_B+g\|_{\rm sp}=|\lambda|$ for all $\lambda\in\FF$ and $g\in\fm$.
\end{proof}
\begin{corollary}\label{cor:aboutval} Let $\FF$ be a field, $A$ be a finite $\FF$-algebra, and set $Z=\Spec A$. Consider a function $f\in \CO(Z)=A$ and the linear operator $L_f\colon \CO(Z)\to \CO(Z)$, defined by  $L_f(a):=fa$.
Then:
\begin{description}
       \item(i) $\CO(Z)=\oplus_{q\in Z}\CO_{Z,q}$,
       \item(ii) the set of eigenvalues of $L_f$ is $\{f(q)\}_{q\in Z}$, and
       \item(iii) if $\FF$ is non-Archimedean and $A$ is equipped with a non-Archimedean norm $\|\cdot\|$ then $\|fe_q\|_{\rm sp}=|f(q)|$ for any $q\in Z$, where $e_q$ denotes the unit element in $\CO_{Z,q}$.
       \end{description}
%
\end{corollary}
\begin{proof}
\begin{description}
       \item(i)  $\dim_\FF A<\infty$ implies $\dim Z=0$ and $\CO(Z)=\oplus_{q\in
Z}\CO_{Z,q}$.
       \item(ii) It is sufficient to show that the operator $L_{f|\CO_{Z,q}}\colon \CO_{Z,q}\to \CO_{Z,q}$ has unique
       eigenvalue $f(q)$. Notice that $fe_q=f(q)e_q+g$, where $g\in \fm_q$ is a nilpotent element.
       Thus $L_{f|\CO_{Z,q}}-f(q)Id_{\CO_{Z,q}}$ is nilpotent, which implies the statement.
       \item(iii) Notice that $fe_q=f(q)e_q+g$, where $g\in \fm_q$; thus
$\|fe_q\|_{\rm sp}=|f(q)|$ by Lemma~\ref{lem-about-non-arc-norms}.
       \end{description}
\end{proof}

\subsection{Symplectic toric manifolds} \label{sec-toric}

\begin{notation*}
Throughout the paper  $M$ denotes a lattice, i.e. a free
abelian
group of finite rank $d$, and  $N=\Hom_\ZZ(M,\ZZ)$ 
its dual lattice. We use the notation
$M_\RR=M\otimes_\ZZ\RR$ and $N_\RR=N\otimes_\ZZ\RR$ for the
corresponding pair of dual vector spaces of dimension $d$. We shall use the notation $T_N$ and $T_M$ for the algebraic tori $T_N=\Spec\,\FF[M]=N\otimes_\ZZ\FF^*$ and $T_M=\Spec\,\FF[N]=M\otimes_\ZZ\FF^*$ over the base field $\FF$.
\end{notation*}
Let $T=M_\RR/M=N\otimes_\ZZ(\RR/\ZZ)$ be the compact torus of
dimension $d$ with lattice of characters $M$ and lattice of
cocharacters $N$. A $2d-$dimensional {\it symplectic toric
manifold} is a closed connected symplectic manifold $(X,\omega)$
equipped with an effective Hamiltonian $T$-action, and a moment
map  $\mu \colon  X \rightarrow Lie(T)^*=M_\RR$ generating (locally) the
$T$-action on $X$. In other words, for any $g\in T$ there is $x\in
X$ such that $g(x)\ne x$, and for any $\xi\in Lie(T)$ and $x\in X$
we have: $d_x\mu(\xi)=\omega(X_\xi,\cdot)$, where $X_\xi$ denotes
the vector field induced by $\xi$ under the exponential map.

By a well known theorem of Atiyah and Guillemin-Sternberg, the
image of the moment map $\Delta := \mu(X) \subset M_\RR$ is the
convex hull of the images of the fixed points of the action. It
was proved by Delzant~\cite{D} that the moment polytope
$\Delta\subset M_\RR$ has the following properties: (i) there are $d$ edges
meeting at every vertex $v$ (simplicity), (ii) the slopes of all
edges are rational (rationality), and (iii) for any vertex $v$ the
set of primitive integral vectors along the edges containing $v$
is a basis of the lattice $M$ (smoothness). Such a polytope is
called {\it a Delzant polytope}. Recall that any polytope can be
(uniquely) described as the intersection of (minimal set of) closed half-spaces with rational slopes.
Namely, there exist $n_1,\dotsc ,n_r\in N=Hom_\ZZ(M,\ZZ)$ and
$\lambda_1,\dotsc ,\lambda_r\in\RR$, where $r$ is the number of facets
(i.e. faces of codimension one) of $\Delta$ such that
\begin{equation} \label{def-Delzant-Poly-I}
\Delta=\{ m \in M_\RR \ | \ (m, n_k)\geq \lambda _k  \ \ {\rm for
\ every} \ k \}.
\end{equation}

Moreover, Delzant gave a complete classification of symplectic
toric manifolds in terms of the combinatorial data encoded by a
Delzant polytope. In~\cite{D} he associated to a Delzant polytope
$\Delta \subset M_\RR$ a closed symplectic manifold
$(X^{2d}_{\Delta},\omega_{\Delta})$ together with a Hamiltonian
$T$-action and a moment map $\mu_\Delta \colon X^{2d}_{\Delta}
\rightarrow M_\RR$ such that $\mu(X^{2d}_\Delta)=\Delta$. He
showed that $(X^{2d}_{\Delta},\omega_{\Delta})$ is isomorphic (as
Hamiltonian
$T$-space) to ($X^{2d},\omega)$, and 
proved that two symplectic toric manifolds are (equivariantly)
symplectomorphic if and only if their Delzant polytopes differ by
a translation and an element of $Aut(M)$. 

The precise relations between the combinatorial data of the
Delzant polytope $\Delta$ and the symplectic structure of $X$ are
as follows: the faces of $\Delta$ of dimension $d'$ are in
one-to-one correspondence with the closed connected equivariant
submanifolds of $X$ of (real) dimension $2d'$, namely to a face
$\alpha$ corresponds the submanifold $\mu^{-1}(\alpha)$. In
particular to facets of $\Delta$ correspond submanifolds of
codimension $2$. Let $z_1,\dotsc ,z_r\in H^2(X,{\mathbb Z})$ be the
Poincar\'e dual of the homology classes of $D_1,\dotsc ,D_r$, where
$D_k$ is the submanifold corresponding to the facet given by
$(m,n_k)=\lambda_k$. Then the cohomology class $[\omega] $ and the
first Chern class $c_1(X)$ are given by
\begin{equation} \label{eq-Chern-omega}  {\frac 1 {2 \pi}}  [\omega]  =
- \sum_{i=1}^r \lambda_k \, z_k, \ \ \ {\rm and} \ \ c_1(X) =
\sum_{i=1}^r z_k \end{equation}

In what follows it would be convenient for us to adopt the
algebraic-geometric point of view of toric varieties which we now
turn to describe.

\subsection{Algebraic Toric Varieties.} \label{subsec-alg-toric}

In this subsection we briefly discuss toric varieties from the
algebraic-geometric point of view.
We refer the reader to the appendix of this paper for the
definitions, and for a more detailed discussion of the notions
that appear below.
For a complete exposition of the subject see Fulton's
book~\cite{F93} and Danilov's survey~\cite{D78}.

Let $\sigma\subset N_\RR$ be a strictly convex, rational,
polyhedral cone. One can assign to $\sigma$ an affine toric
variety $X_\sigma=\Spec\,\FF[M\cap \check{\sigma}]$, where
$\check{\sigma}\subset M_\RR$ is the dual cone and $\FF[M\cap
\check{\sigma}]$ is the corresponding commutative semigroup
algebra. If $\tau\subseteq\sigma$ is a face then
$X_\tau\hookrightarrow X_\sigma$ is an open subvariety. In
particular, since $\sigma$ is strictly convex, the affine toric
variety $X_\sigma$ contains the torus
$X_{\{0\}}=\Spec\,\FF[M]=N\otimes_\ZZ\FF^*=T_N$ as a dense open
subset. Furthermore, the action of the torus on itself extends to
the action on $X_\sigma$.

Recall that a collection $\Sigma$ of strictly convex, rational,
polyhedral cones in $N_\RR$ is called a {\it fan} if the following
two conditions hold:
\begin{enumerate}
\item If $\sigma\in \Sigma$ and $\tau\subseteq\sigma$ is a face
then $\tau\in \Sigma$. \item If $\sigma,\tau\in \Sigma$ then
$\sigma\cap\tau$ is a common face of $\sigma$ and $\tau$.
\end{enumerate}
A fan $\Sigma$ is called {\it complete} if
$\cup_{\sigma\in\Sigma}\sigma=N_\RR$. One-dimensional cones in $\Sigma$ are called {\it rays}.
\begin{notation*}
The set of cones of
dimension $k$ in $\Sigma$ is denoted by $\Sigma^k$, and the
primitive integral vector along a ray $\rho$ is denoted by
$n_\rho$.
\end{notation*}

Given a (complete) fan $\Sigma \subset N_\RR$ one can construct a
(complete) toric variety $X_\Sigma=\cup_{\sigma\in\Sigma}X_\sigma$
by gluing $X_\sigma$ and $X_\tau$ along $X_{\sigma\cap\tau}$.
Recall that $X_\Sigma$ has only orbifold singularities if and only
if all the cones in $\Sigma$ are simplicial (in this case it is
called quasi-smooth); and $X_\Sigma$ is smooth if and only if for
any cone $\sigma\in\Sigma$ the set of primitive integral vectors
along the rays of $\sigma$ forms a part of a basis of the lattice
$N$.

The torus $T_N$ acts on $X_\Sigma$ and decomposes it into a
disjoint union of orbits. To a cone $\sigma\in\Sigma$ one can
assign an orbit $O_\sigma\subset X_\sigma$, canonically isomorphic
to $\Spec\,\FF[M\cap\sigma^\bot]$. This defines a one-to-one
order reversing correspondence between the cones in $\Sigma$ and
the orbits in $X_\Sigma$. In particular orbits of codimension one
correspond to rays $\rho\in\Sigma$ and we denote their closures by
$D_\rho$. Thus $\{D_\rho\}_{\rho\in\Sigma^1}$ is the set of
$T_N$-equivariant primitive Weil divisors on the variety
$X_\Sigma$. We remark that the set $\{D_\rho\}_{\rho\in\Sigma^1}$
coincides with the set $\{D_i\}_{1\le i \le r}$ in the setting of
the previous subsection.

For a polytope $\Delta\subset M_\RR$ of dimension $d$ one can
assign a complete fan $\Sigma$ and a piecewise linear strictly
convex function $F$ on $\Sigma$ in the following way: To a face
$\gamma\subseteq\Delta$ we assign the cone $\sigma$ being the dual
cone to the inner angle of $\Delta$ at $\gamma$ (see~\cite{D78} \S
5.8); and if $m$ is a vertex of $\Delta$ and $\sigma_m\in\Sigma$
is the corresponding cone then $F_{|_{\sigma_m}}:=m$. Vice versa,
to a pair $(\Sigma, F)$ one can assign a polytope
\begin{equation} \label{def-Delzant-Poly-II} \Delta_F=\{m\in M_\RR\,|\, (m,n_\rho)\ge F(n_\rho), \ {\rm
for \ every \ \rho} \}.
\end{equation}
This gives a bijective
correspondence between polytopes of dimension $d$ in $M_\RR$ and
pairs $(\Sigma, F)$ as above.
 It is
known (see the Appendix for details) that choosing a piecewise
linear strictly convex function $F$ on $\Sigma$ as above is
equivalent to introducing a symplectic structure $\omega$ on
$X_\Sigma$ (such that the torus action is Hamiltonian) together with a moment map.
Under this identification, the polytope $\Delta_F$
$(\ref{def-Delzant-Poly-II})$ coincides with the polytope $\Delta$
$(\ref{def-Delzant-Poly-I})$ of the symplectic manifold
$(X_{\Sigma}, \omega)$ with the corresponding moment map. As
mentioned before, in what follows, it will be more convenient for
us to adopt the algebraic point of view and to consider the pair
$(X_\Sigma, F)$ instead of the symplectic toric manifold
$(X,\omega)$.

For a real/rational/integral piecewise linear function $F$ on a
fan $\Sigma$ one can associate a $T_N-$equivariant
$\RR/\QQ/\ZZ-$Cartier divisor
$D=-\sum_{\rho\in\Sigma^{1}}F(n_\rho)D_\rho$. Moreover, any
$\RR/\QQ/\ZZ-$ Cartier divisor is equivalent to a $T_N-$equivariant
$\RR/\QQ/\ZZ-$Cartier divisor of this form. Integral $T_N-$equivariant
Cartier divisors are called $T-$divisors. It is well known
that strictly convex piecewise linear functions $F$ correspond to ample divisors.
Moreover, if $F$ is integral then the Cartier divisor
$D=-\sum_{\rho\in\Sigma^{1}}F(n_\rho)D_\rho$ corresponds to an
invertible sheaf (i.e., a line bundle) $\CL=\CO_{X_\Sigma}(D)$
together with a trivialization $\phi\colon \CL_{|_{T_N}}\to
\CO_{T_N}$ defined up-to the natural action of $\FF^*$.
\begin{remark} \label{rem-on-sections} {\rm For an integral
function $F$ as above, the trivialization $\phi$ identifies the global sections of $\CO_{X_\Sigma}(D)$ with functions on $T_N$, furthermore the following holds
\begin{equation} H^0(X_\Sigma, \CO_{X_\Sigma}(D)) \simeq
\Span\{x^m\}_{m\in\Delta_F\cap M}\subset \CO(T_N). \end{equation}
}\end{remark}

Let $F$ be an integral strictly convex piecewise linear function on $\Sigma$. Recall that the orbits in $X_\Sigma \subset N_\RR$ are in
one-to-one order reversing correspondence with the cones in
$\Sigma$, hence they are in one-to-one order preserving
correspondence with the faces of $\Delta_F$. Let $\gamma \subset
M_\RR$ be a face of $\Delta_F$, let $\sigma_\gamma\in\Sigma$ be
the corresponding cone, and let $V=\overline{O}_{\sigma_\gamma}\subset X_\Sigma$ be the closure of the corresponding orbit. Then
$V$ has a structure of a toric variety with respect to the action of the torus $\Spec\,\CCC[M\cap \sigma_\gamma^\perp]$, and the
restriction $\CL_V$ of $\CL$ to $V$ is an ample line bundle on $V$; however, $\CL_V$ has no distinguished trivialization. To define a trivialization one must pick an integral point $p$ in the affine space $\Span(\gamma)$ (e.g. a vertex of $\gamma$) and this defines an isomorphism between $\CL_V$ and the line bundle associated to the polytope $\gamma-p\subset \sigma_\gamma^\perp$.

\subsection*{Toric Fano Varieties and Reflexive Polytopes.}
Let $\Delta \subset M_\RR$ be a polytope containing $0$ in its interior.
The dual polytope  $\Delta^* \subset N_\RR$ is defined to be  $$\Delta^* = \{ n \in N_\RR \ | \ (m,n)\geq -1, \ {\rm for \ every} \ m \in \Delta \}.$$ Notice that
its vertices are precisely the inner normals to the facets of
$\Delta$. The polytope $\Delta\subset M_\RR$ is called reflexive
if (i) $0$ is contained in its interior, and (ii) both $\Delta$
and $\Delta^*$ are integral polytopes. Note that if $\Delta$ is
reflexive then $0$ is the only integral point  in its interior. It
is not hard to check (cf.~\cite{Bat2}) that $\Delta$ is reflexive
if and only if its dual $\Delta^*$ is reflexive.

A complete algebraic variety is called Fano if its anti-canonical class is
Cartier and ample. Recall
that if $X_\Sigma$ is Fano and $K=-\sum D_\rho = -\sum
F_K(n_\rho)D_\rho$ is the standard canonical $T$-divisor then
$\Delta_{-F_K}=\Delta_{F_{-K}}$ is reflexive, here $F_K$ is a
 piecewise linear function defined by the following property:
$F_K(n_\rho)=1$ for any $\rho \in \Sigma^1$. Moreover, if $\Delta$
is reflexive then there exists a unique toric Fano variety
$X_\Sigma$ such that $\Delta=\Delta_{F_K}$, where $K=-\sum
D_\rho$, and $F_K$ is as above.

Let $X_\Sigma$ be a toric Fano variety, $\Delta=\Delta_{F_{-K}}$ be
the reflexive polytope assigned to the anticanonical divisor $-K=\sum
D_\rho$, and $\Delta^*$ be the dual reflexive polytope. Consider
the dual toric Fano variety $X_\Sigma^*=X_{\Sigma^*}$ assigned to
the polytope $\Delta^*$. Then the fan $\Sigma$ coincides with the
fan over the faces of $\Delta^*$, and the fan $\Sigma^*$ is the
fan over the faces of $\Delta$.

Let now $X=X_\Sigma$ and $X^*=X_\Sigma^*$ be a pair of dual toric
Fano varieties, and assume that $X$ is smooth. Then any maximal
cone in $\Sigma$ is simplicial, and is generated by a basis of $N$;
hence the facets of the dual  polytope $\Delta^*$ are basic
simplexes. Thus the irreducible components of the complement of
the big orbit in $X^*$ are isomorphic to $\PP^{d-1}$.
Furthermore, the restriction of the anticanonical linear system
$\CO_{X^*}(-K_{X^*})$ to such a component is isomorphic
to the anti-tautological line bundle $\CO_{\PP^{d-1}}(1)$.



\begin{remark}
Before we finish this subsection we wish to recall the
following two facts: (i) (see~\cite{F93} section 3.2) the Euler
characteristic of a quasi-smooth complete toric variety is equal
to $|\Sigma^d|$, and (ii) (Kushnirenko's theorem, a particular
case of Bernstein's theorem - see~\cite{F93} section 5.3) if $D$
is an ample $T$-divisor on a toric variety $X_\Sigma$, and
$\Delta\subset M_\RR$ is the corresponding polytope, then the
intersection number $D^d$ is given by $D^d=d!{\rm
Volume}(\Delta)$, where the volume is relative to the lattice $M$.
\end{remark}

\section{The Quantum Cohomology} \label{subsec-qh}
Below are three equivalent descriptions of the quantum cohomology
of Fano toric varieties.
\subsection{Symplectic Definition} \label{qh-sym-def}
We start with a symplectic definition of the quantum homology (and
cohomology) of a $2d$-dimensional symplectic manifold $(X,\omega)$,
using Gromov-Witten invariants. We refer the reader to~\cite{MS} and
the references within for a more detailed exposition. For
simplicity, throughout the text we assume that $(X,\omega)$ is
semi-positive manifold (see e.g. Subsection 6.4 in~\cite{MS}). The
class of symplectic toric Fano manifolds is a particular example.

By abuse of notation, we write $\omega(A)$ and $c_1(A)$ for the
results of evaluation of the cohomology classes $[\omega]$ and $c_1$
on $A \in H_2(X;{\mathbb Z})$. Here $c_1 \in H^2(X;{\mathbb Z})$
denotes the first Chern class of $X$.
We denote by $\KK^{\scriptscriptstyle \downarrow}$ 
the field of generalized Laurent series over $\CCC$.
More precisely,
\begin{equation} \label{def-the-field-K} \KK^{\scriptscriptstyle \downarrow} = \Bigl \{
\sum_{\lambda \in \RR}  a_{\lambda} s^{\lambda} \ | \ a_{\lambda}
\in \CCC, \ {\rm and} \ \{ \lambda \, | \, a_\lambda \neq 0 \} \
{\rm is \ discrete \ and \ bounded \ above \ in \ \RR } \Bigr \}
\end{equation}
Similarly, we define $\KK^{\scriptscriptstyle \uparrow}$ to be the
field of generalized Laurent series where the set $\{\lambda \, |
\, a_\lambda \neq 0 \}$ is discrete and bounded from below in
$\RR$.
In the definition of the quantum homology we shall use
the Novikov ring $\Lambda^{\scriptscriptstyle
\downarrow}:={\KK^{\scriptscriptstyle \downarrow}}[q,q^{-1}]$.
and in the definition of the
quantum cohomology we use the ``dual'' ring
$\Lambda^{\scriptscriptstyle \uparrow}:={\KK^{\scriptscriptstyle
\uparrow}}[q,q^{-1}]$. By setting $\deg(s)=0$ and $\deg(q)=2$ we introduce the structure of graded rings on $\Lambda^{\scriptscriptstyle
\downarrow}$ and $ \Lambda^{\scriptscriptstyle \uparrow}$.

As a graded module 
the quantum homology (cohomology) algebra of $(X,\omega)$ is defined to be
$$QH_*(X,\omega) = H_*(X,{\mathbb Q}) \otimes_{\mathbb Q} \Lambda^{\scriptscriptstyle \downarrow},
\ \ QH^*(X,\omega) = H^*(M,{\mathbb Q}) \otimes_{\mathbb Q}
\Lambda^{\scriptscriptstyle \uparrow}.$$
The grading on $QH_*(X,\omega)$
(respectively on $QH^*(X,\omega)$) is given by ${\rm deg}(a \otimes
s^{\lambda}q^{j}) = {\rm deg}(a) + 2j$, where ${\rm deg}(a)$ is
the standard degree of the class $a$ in the homology (cohomology)
of $(X,\omega)$. Next we define the quantum product (cf~\cite{MS}). We
start with the quantum homology $QH_*(X,\omega)$. For $a \in
H_i(X,{\mathbb Q})$ and $b \in H_j(X,{\mathbb Q})$, define $(a\otimes 1)
* (b\otimes 1) \in QH_{i+j-2d}(X,\omega)$ by
$$ (a\otimes 1) * (b\otimes 1) = \sum_{A \in H_2^S(X)} (a * b)_A \otimes s^{-\omega(A)}
q^{-c_1(A)},$$ where $(a * b)_A \in H_{i+j-2d+2c_1(A)}(M,{\mathbb
Q})$ is defined by the requirement that
$$ (a*b)_A  \circ c = GW_A(a,b,c), \ \ {\rm for \ all} \ c \in H_*(X,{\mathbb Q}).$$
Here $\circ$ is the usual intersection index and $GW_A(a,b,c)$
denotes the Gromov-Witten invariant that, roughly speaking, counts
the number of pseudo-holomorphic spheres representing the class $A$
and intersecting with generic representative of each $a,b,c \in
H_*(X,{\mathbb Q})$ (see e.g.~\cite{MS},~\cite{RT1},
and~\cite{RT2} for the precise definition). The product $*$ is
extended to the whole $QH_*(X,\omega)$ by linearity over
$\Lambda^{\scriptscriptstyle \downarrow}$. Thus, one gets a
well-defined commutative, associative product operation $*$ respecting the grading
on $QH_*(X,\omega)$, which is a deformation of the classical cap-product
in singular homology
(see~\cite{MS},~\cite{RT1},~\cite{RT2}~\cite{Liu}, and~\cite{W}).
Note that the fundamental class $[X]$ is the unity with respect to
the quantum multiplication $*$, and that $QH_*(X,\omega)$ is a
finite-rank module over $\Lambda^{\scriptscriptstyle \downarrow}$.
Moreover, if $a,b \in QH_*(X,\omega)$ have graded degrees $deg(a)$ and
$deg(b)$ respectively, then $ deg((a\otimes 1) * (b\otimes 1)) = deg(a) + deg(b) - 2d.$

Due to some technicalities and although the above definition is
more geometric, in what follows we shall mainly use the quantum
cohomology. The quantum product in this case is defined using
Poincar\'{e} duality i.e., for $\alpha,\beta \in H^*(X,{\mathbb
Q})$ with Poincar\'{e} duals $a = {\rm PD}(\alpha), b = {\rm
PD}(\beta)$ we define
$$ (\alpha\otimes 1) * (\beta\otimes 1)  = {\rm PD}_q (a * b) := \sum_{A \in H_2^S(X)} {\rm PD}((a * b)_A) \otimes s^{\omega(A)}
q^{c_1(A)},$$ where the quantum Poincar\'{e} dual map ${\rm PD}_q \colon
QH^*(X,\omega) \rightarrow QH_*(X,\omega)$ is the obvious variation
of the standard Poincar\'{e} dual given by ${\rm PD}_q(\alpha
\otimes s^{\lambda}q^{j}) = {\rm PD}(\alpha) \otimes
s^{-\lambda}q^{-j}$.

As mentioned in the introduction, our main object of study is the
subalgebra $QH_{2d}(X,\omega)$, which is the graded component of degree
$2d$ in the quantum homology algebra $QH_*(X,\omega)$. It is not hard to
check that it is a commutative algebra of finite rank over the
field ${\KK}^{\scriptscriptstyle \downarrow}$. The above mentioned
(quantum) Poincar\'{e} duality induces an isomorphism between the
quantum homology and cohomology (see~\cite{MS} remark 11.1.16).
Hence, in what follows we will work with the algebra $QH^0(X,\omega)$
over the field ${\KK}^{\scriptscriptstyle \uparrow}$ instead of
the algebra  $QH_{2d}(X,\omega)$ over ${\KK}^{\scriptscriptstyle
\downarrow}$.
\begin{convention*}
From this point on we set $\KK := {\KK}^{\scriptscriptstyle
\uparrow}$ and use the Novikov ring $\Lambda := \KK[q,q^{-1}].$
\end{convention*}
\begin{remark}\label{rem:normOnKandQH} Notice that the field $\KK$ is a non-Archimedean field with respect
to the non-Archimedean norm $\left|\sum a_{\lambda}
s^{\lambda}\right|:=10^{-\inf\{\lambda\,|\,a_\lambda\ne 0\}}.$ It
is known that $\KK$ is algebraically closed. Notice also that the
map $\|\cdot\|\colon QH^*(X,\omega)\to \RR_+$ defined by
$\|\sum_{\lambda, j} a_{\lambda j}s^\lambda
q^j\|=10^{-\inf\{\lambda\,|\,\exists\, a_{\lambda j}\ne 0\}}$,
where $a_{\lambda j}\in H^*(X,\CCC)$, is a non-Archimedean norm on
the quantum cohomology algebras $QH^*(X,\omega)$ and
$QH^0(X,\omega)$.
\end{remark}


\subsection{Batyrev's Description of the Quantum Cohomology}

In~\cite{Bat1}, Batyrev proposed a combinatorial description of
the quantum cohomology algebra of toric Fano manifolds, using a
``quantum'' version of the ``classical" Stanley-Reisner ideal.
This was later proved by Givental in~\cite{Giv},~\cite{Giv1}. For
a different approach to the proof we refer the reader to
McDuff-Tolman~\cite{MT} and Cieliebak-Salamon~\cite{CS}.

Before describing Batyrev's work let us first briefly recall the
definition of the classical cohomology of toric Fano manifolds.
The complete details can be found in~\cite{D78} \S 10,11,12,
and~\cite{F93} section 3.2 and Chapter 5.

Let $\Sigma$ be a simplicial fan, and let $X_\Sigma$ be the
corresponding toric variety over $\CCC$. It is known that any cohomology class
has an equivariant representative. Thus,
$H^{2k}(X_\Sigma,\QQ)$ is generated as a vector space by the
closures of $k$-dimensional orbits. Notice that any such closure
$V$ is an intersection of some equivariant divisors $D_\rho$ with
appropriate multiplicity that depends on the singularity of the
$X_\Sigma$ along $V$. To be more precise, if
$V=\overline{O_\sigma}$, $\sigma\in\Sigma^k$, and
$\rho_1,\dotsc ,\rho_k$ are the rays of $\sigma$ then $V=\frac{1}{{\rm
mult}(\sigma)}\prod_{i=1}^kD_{\rho_i}$, where ${\rm mult}(\sigma)$
denotes the covolume of the sublattice spanned by
$n_{\rho_1},\dotsc ,n_{\rho_k}$ in the lattice $\Span(\sigma)\cap N$.
Thus we have a surjective homomorphism of algebras
$\psi\colon \QQ[z_\rho]_{\rho\in\Sigma^1}\to
H^{2*}(X_\Sigma,\QQ)$, where $\QQ[z_\rho]_{\rho\in\Sigma^1}$ is
the polynomial algebra in free variables $z_\rho$ indexed by the rays $\rho\in\Sigma^1$.

Let $x^m\in\CCC[M]$ be a rational function on $X_\Sigma$. Then
$div(x^m)=\sum_{\rho\in\Sigma^1}(m,n_\rho)D_\rho$. Thus
$\sum_{\rho\in\Sigma^1}(m,n_\rho)z_\rho\in \Ker(\psi)$ for any
$m\in M$. We denote by
$P(X_\Sigma)\subset\QQ[z_\rho]_{\rho\in\Sigma^1}$ the ideal
generated by $\sum_{\rho\in\Sigma^1}(m,n_\rho)z_\rho$, $m\in M$.
Notice that if $\rho_1,\dotsc ,\rho_k$ do not generate a cone in
$\Sigma$ then $\cap_{i=1}^kD_{\rho_i}=\emptyset$, and thus
$\prod_{i=1}^kz_{\rho_i}\in \Ker(\psi)$. We denote by
$SR(X_\Sigma)\subset\QQ[z_\rho]_{\rho\in\Sigma^1}$ the
Stanley-Reisner ideal, i.e. the ideal generated by
$\prod_{i=1}^kz_{\rho_i}$ where $\rho_1,\dotsc ,\rho_k$ do not
generate a cone in $\Sigma$. It is well known that
$\Ker(\psi)=P(X_\Sigma)+SR(X_\Sigma)$, and hence
$$
H^{2*}(X_\Sigma,\QQ)=\frac{\QQ[z_\rho]_{\rho\in\Sigma^1}}{P(X_\Sigma)+SR(X_\Sigma)}.$$

We turn now to Batyrev's description of the quantum cohomology.
 We say that the set of rays
$\rho_1,\dotsc ,\rho_k$ is a {\it primitive collection} if
$\rho_1,\dotsc ,\rho_k$ do not generate a cone in $\Sigma$ while any
proper subset does generate a cone in $\Sigma$. Notice that the set
of monomials $\prod_{i=1}^k z_{\rho_i}$ assigned to primitive
collections forms a minimal set of generators of $SR(X_\Sigma)$. The
quantum version of the Stanley-Reisner ideal $QSR(X_\Sigma)$ is
generated by the quantization of the minimal set of generators
above.

More precisely, assume that we are given a smooth Fano toric
variety $X_\Sigma$, and a piecewise linear strictly convex
function $F$ on $\Sigma$ defining an ample $\RR$-divisor on
$X_\Sigma$. Let $C$ be a primitive collection of rays. Then
$\sum_{\rho\in C}n_\rho$ belongs to a cone $\sigma_C\in\Sigma$,
and we assume that $\sigma_C$ is the minimal cone containing it.
It is not hard to check that $\sigma_C$ does not contain $\rho$
for all $\rho\in C$ (cf~\cite{Bat1}). Since $X_\Sigma$ is smooth,
$\sum_{\rho\in C}n_\rho=\sum_{\rho\subseteq\sigma_C}a_\rho
n_\rho$, where $a_\rho$ are strictly positive integers. We define
the quantization of the generator $\prod_{\rho\in C} z_\rho$ to be
$$\prod_{\rho\in
C}q^{-1}s^{-F(n_\rho)}z_\rho-\prod_{\rho\subseteq\sigma_C}(q^{-1}s^{-F(n_\rho)}z_\rho)^{a_\rho}$$
The quantum version of $SR(X_\Sigma)$ is the ideal
$QSR(X_\Sigma,F)\subset \Lambda[z_\rho]_{\rho\in\Sigma^1}$
generated by the quantization of the minimal set of generators. We define Batyrev's
quantum cohomology to be
$$QH^*_B(X_\Sigma,F;\Lambda):=\frac{\Lambda[z_\rho]_{\rho\in\Sigma^1}}{P(X_\Sigma)+QSR(X_\Sigma,F)},$$
and
$$QH^*_B(X_\Sigma,F;\KK):=QH^*_B(X_\Sigma,F;\Lambda)\otimes_{\Lambda}\Lambda \, / \, \langle q-1 \rangle .$$ As mentioned
above, the following result was originally proposed by Batyrev
\cite{Bat1} and proved by Givental \cite{Giv,Giv1}. For a proof
using notation and conventions similar to ours see~\cite{MT}. Recall
that $(X,\omega)$ and $(X_\Sigma,F)$ represents the same symplectic
toric Fano manifold as explained in
Subsection~\ref{subsec-alg-toric}} above.

\begin{theorem} \label{Batyrev's-description-of-QH} For a symplectic toric Fano manifold $(X,\omega)=(X_\Sigma,F)$
there is a ring isomorphism \begin{equation} \label{eq-QH=QHB}
QH^*(X,\omega) \simeq QH^*_B(X_\Sigma,F;\Lambda) \end{equation}
\end{theorem}

 We wish to remark that the identification~$(\ref{eq-QH=QHB})$ may fail without the Fano
 assumption (see~\cite{CK99} example 11.2.5.2 and~\cite{MT}).

\subsection{The Landau-Ginzburg Superpotential}\label{sec:LGP}

Here we present an analytic description of the quantum cohomology
algebra for symplectic toric Fano varieties which arose from the
study of the corresponding Landau-Ginzburg model in
Physics~\cite{LVW},~\cite{V1},~\cite{HV}. We will follow the works
of Batyrev~\cite{Bat1}, Givental~\cite{Giv}, Hori-Vafa~\cite{HV},
Fukaya-Oh-Ohta-Ono~\cite{FOOO}, and describe an isomorphism
between the quantum cohomology algebra of a symplectic toric Fano
manifold $X$ and the Jacobian ideal of the superpotential
corresponds to the Landau-Ginzburg mirror model of $X$.

Let $X_\Sigma$ be a smooth Fano toric variety, and let $F$ be a
piecewise linear strictly convex  function on $\Sigma$ defining an
ample $\RR$-divisor on $X_\Sigma$. Consider the Landau-Ginzburg
superpotential
$$W_{F,\Sigma}:=\sum_{\rho\in\Sigma^1}s^{F(n_\rho)}x^{n_\rho}$$
defined on the torus $\Spec\KK[N]$. This function can be considered
also as a section of the anti-canonical line bundle on the dual
toric Fano variety $X_\Sigma^*$ over the field $\KK$ (see
Remark~\ref{rem-on-sections}). One assigns to such a function the
Jacobian ring $\KK[N]/J_{W_{F,\Sigma}}$, where $J_{W_{F,\Sigma}}$
denotes the Jacobian ideal, i.e. the ideal generated by all partial
(log-)derivatives of $W_{F,\Sigma}$.

\begin{proposition}\label{prop:LGmodel} If $(X_\Sigma,F)$ is a rational
smooth symplectic toric Fano variety, and $W_{F,\Sigma}$ as above
then
$$  QH^*(X,\omega) \cong QH_B^*(X_\Sigma,F;\Lambda)\cong \Lambda[N]/J_{W_{F,\Sigma}},$$
and in particular
$$  QH^0(X,\omega) \cong QH_B^*(X_\Sigma,F;\KK)\cong \KK[N]/J_{W_{F,\Sigma}}.$$
\end{proposition}
For the proof of Proposition~\ref{prop:LGmodel} we shall need the following lemma.

\begin{lemma}\label{lemma:jscheme} Let $X=X_\Sigma$ be a smooth toric Fano variety over the base field $\KK$,
$F$ be a piecewise linear strictly convex function on $\Sigma$,
and $W=W_{F,\Sigma}$ be the corresponding Landau-Ginzburg
superpotential, or more generally, a section $\sum_{\rho\in\Sigma^1}b_\rho x^{n_\rho}$ of the anticanonical bundle on $X^*$ with all $b_\rho\ne 0$. Let $Z_W\subset X^*=X_\Sigma^*$ be the subscheme defined by the ideal
sheaf $\CJ_W\subset\CO_{X^*}$, where $\CJ_W(-K_{X^*})\subset\CO_{X^*}(-K_{X^*})$ is generated by all log-derivatives of
$W$. Then $Z_W$ is a projective subsheme of the big orbit $T_M\subset X^*$ of degree $|\Sigma^d|$. In particular it is zero dimensional, $\CO(Z_W)=\KK[N]/J_W$, and $\dim \CO(Z_W)=|\Sigma^d|$.
\end{lemma}

\begin{proof}[{\bf Proof of Lemma~\ref{lemma:jscheme}}]
Since $X_\Sigma$ is smooth each irreducible component of
$X_\Sigma^*\setminus T_M$ is isomorphic to $\PP^{d-1}$. Recall
that such components are in one-to-one correspondence with the
rays of the dual fan $\Sigma^*$, or equivalently with the maximal
cones in $\Sigma$. Furthermore, if $\sigma\in\Sigma^d$ is a cone
and $D^*_\sigma\simeq\PP^{d-1}$ is the corresponding component
then the restriction of
the anticanonical linear system to such a component
$\CO_{X^*}(-K_{X^*})\otimes\CO_{D^*_\sigma}$ is isomorphic
to ${\cal O}_{\PP^{d-1}}(1)$, and the homogeneous
coordinates on $D^*_\sigma$ are naturally parameterized by the
rays $\rho\subset\sigma$. We denote these coordinates by $y_\rho$.

We consider $W$ and its log-derivatives as sections of
$\CO_{X^*}(-K_{X^*})$. Then,
$\partial_mW=\sum_{\rho\in\Sigma^1}(m,n_\rho)b_\rho x^{n_\rho}$
and its restriction to  $D^*_\sigma$ is given by
$\sum_{\rho\subset\sigma}(m,n_\rho)b_\rho y_\rho$. Clearly the set of
these equations for $m\in M$ has no common roots, hence
$Z_W\subset T_M$. But $Z_W\subset X_\Sigma^*$ is closed, hence a
projective scheme. Thus $Z_W$ is zero dimensional.

By definition $Z_W$ is the scheme-theoretic intersection of $d$
sections of $\CO_{X^*}(-K_{X^*})$, hence by  Kushnirenko's theorem
$$\deg Z_W=(-K_{X^*})^d=d!Volume(\Delta^*)=d!\sum_{\sigma\in\Sigma^d}Volume(\Delta^*\cap\sigma)=|\Sigma^d|,$$
since $\Delta^*\cap\sigma$ is a primitive simplex for any
$\sigma\in\Sigma^d$.
\end{proof}

\begin{proof}[{\bf Proof of Proposition~\ref{prop:LGmodel}.}]
Consider the natural homomorphism
\begin{equation} \label{from-Batyrev-to-LG} \psi\colon \Lambda[z_\rho]_{\rho\in\Sigma^1}\to \Lambda[N], \ \ {\rm
defined \ by} \ \psi(z_\rho)=qs^{F(n_\rho)}x^{n_\rho}.
\end{equation} Since $X_\Sigma$ is smooth and projective (hence
complete) the fan $\Sigma$ is complete, and any $n\in N$ is an
integral linear combination of vectors $n_\rho$, $\rho\in\Sigma^1$.
Thus, $\psi$ is surjective.

Next we claim that the quantum Stanley-Reisner ideal
$QSR(X_\Sigma,F)$ lies in the kernel of $\psi$. Indeed, let $C$ be
a primitive collection and let $$\prod_{\rho\in
C}q^{-1}s^{-F(n_\rho)}z_\rho-\prod_{\rho\subseteq\sigma_C}(q^{-1}s^{-F(n_\rho)}z_\rho)^{a_\rho},$$
be the corresponding quantum generator. It follows from the
definition of $\psi$ that:
$$\psi\Bigl(\prod_{\rho\in C}q^{-1}s^{-F(n_\rho)}z_\rho-\prod_{\rho\subseteq\sigma_C}
(q^{-1}s^{-F(n_\rho)}z_\rho)^{a_\rho}\Bigr)=x^{\sum_{\rho\in
C}n_\rho}-x^{\sum_{\rho\subseteq\sigma_C}a_\rho n_\rho}=0.$$
Moreover, $\psi$ sends the ideal $P(X_\Sigma)$ into $J_{W_{F,\Sigma}}$. Indeed,
let $\sum_{\rho\in\Sigma^1}(m,n_\rho)z_\rho, \ m\in M$ be a
generator of $P(X_\Sigma)$. Then:
$$\psi(\sum_{\rho\in\Sigma^1}(m,n_\rho)z_\rho)=q\sum_{\rho\in\Sigma^1}(m,n_\rho)s^{F(n_\rho)}x^{n_\rho}=q\partial
log_mW_{F,\Sigma}\in J_{W_{F,\Sigma}}.$$ Thus, $\psi$ defines a {\it surjective} homomorphism
$QH_B^*(X_\Sigma,F;\Lambda)\to \Lambda[N]/J_{W_{F,\Sigma}}$.

Notice that both algebras $QH_B^*(X_\Sigma,F;\Lambda)$ and
$\Lambda[N]/J_{W_{F,\Sigma}}$ are {\it free} modules over $\Lambda$, and thus
to complete the proof all we need to do is to compare the ranks.
On one side:
$$\rank_{\Lambda}QH_B^*(X_\Sigma,F;\Lambda)=\dim_\KK H^*(X_\Sigma,\KK)=\chi(X_\Sigma)=|\Sigma^d|.$$
On the other side the rank of $\Lambda[N]/J_{W_{F,\Sigma}}$ over $\Lambda$ is
equal to $\dim_\KK\KK[N]/J_{W_{F,\Sigma}}$, which by Lemma~\ref{lemma:jscheme}
equals $|\Sigma^d|$. The proof is now complete.
\end{proof}

\begin{lemma}\label{lemma:rednondeg}
Let $X$, $X^*$, and $W=\sum_{\rho\in\Sigma^1}b_\rho x^{n_\rho}$ be as in Lemma \ref{lemma:jscheme}. Then the support of $Z_W$ coincides with the set of critical points of the function $\sum_{\rho\in\Sigma^1}b_\rho x^{n_\rho}$ on the torus $T_M$. Furthermore, a critical point $p$ is non-degenerate if and only if the scheme $Z_W$ is reduced at $p$.
\end{lemma}
\begin{proof}
We already proved in Lemma \ref{lemma:jscheme} that $Z_W$ is a zero-dimensional subscheme of the torus $T_M$. Thus $p\in Z_W\subset T_M$ if and only if all log-derivatives of $W$ vanish at $p$ if and only if $p$ is a critical point of $W$. Notice that $p$ is a non-degenerate critical point of $W$ if and only if the Hessian is non-degenerate at $p$, or equivalently, if and only if the differentials of the log-derivatives of $W$ generate the cotangent space $T^*_pT_M$. It remains to show that the latter condition is equivalent to the following: the log-derivatives of $W$ generate the maximal ideal of $p\in T_M$ locally, i.e. $\fm_p=J_{W,p}=J_W\CO_{T_M,p}$, where $\fm_p\subset \CO_{T_M,p}$ denotes the maximal ideal. Clearly if $\fm_p=J_{W,p}$ then the differentials of the log-derivatives generate\footnote{Recall that if $f\in \CO_{T_M,p}$ then $d_pf$ is nothing but the class of $f-f(p)$ modulo $\fm_p^2$.} $T^*_pT_M=\fm_p/\fm_p^2$. To prove the opposite direction we will need Nakayama's lemma. Indeed, if the differentials of the log-derivatives of $W$ generate $T^*_pT_M$ then $\fm_p=J_{W,p}+\fm_p^2$, thus $\fm_p\cdot (\fm_p/J_{W,p})=\fm_p/J_{W,p}$, hence, by Nakayama's lemma, $\fm_p/J_{W,p}=0$, or equivalently $\fm_p=J_{W,p}$.
\end{proof}
\begin{corollary}\label{cor:jschemeGen}
For $X=X_\Sigma$, $X^*$, $W=\sum_{\rho\in\Sigma^1}b_\rho x^{n_\rho}$, and $Z_W\subset X^*=X_\Sigma^*$ as in the lemma the following hold:
\begin{description}
       \item(i) $\CO(Z_W)$ is semisimple  if and only if $W$ has only non-degenerate critical points.
       \item(ii) $\CO(Z_W)$ contains a field as a direct summand if and only if $W$ has a non-degenerate critical point.
       \end{description}
\end{corollary}

\section{Proof of The Main Results} \label{prf-main-results}

In this section we prove our main results. We start with
Theorem~\ref{thmin-generic} which follows from the quantum
Poincar\'e duality described in Subsection~\ref{qh-sym-def} and the
following theorem:

\begin{theorem} \label{thm-generic} Let $X_\Sigma$ be a smooth toric Fano variety.
Then for a generic choice of a toric symplectic form $\omega$ on
$X_\Sigma$ the quantum cohomology $QH^0(X_\Sigma, \omega)$ is
semisimple.
\end{theorem}

The proof follows the arguments in~\cite{I} Corollary 5.12,
and~\cite{FOOO} Proposition 7.6.

\begin{proof}[{\bf Proof of Theorem~\ref{thm-generic}}]

 Let $X^*=X_\Sigma^*$ be the dual Fano toric variety and let $\CO_{X^*}(-K_{X^*})$
  be the anti-canonical linear system. Following Remark~\ref{rem-on-sections} we consider the subspace of
  sections $\Span\{x^{n_\rho}\}_{\rho\in\Sigma^1}\subset H^0(X^*, \CO_{X^*}(-K_{X^*}))$.
  It has codimension one since $X_\Sigma$ is Fano and smooth,
  moreover $H^0(X^*, \CO_{X^*}(-K_{X^*}))$ is generated by $\Span\{x^{n_\rho}\}$
  and the section $x^0$.

Consider a strictly convex piecewise linear function $F$ and the
associated potential
$W_{F,\Sigma}=\sum_{\rho\in\Sigma^1}s^{F(n_\rho)}x^{n_\rho}$. Let
$Z_{W_{F,\Sigma}}$ be the subscheme of $X^*$ defined by the log-derivatives of
$W_{F,\Sigma}$ as in Lemma \ref{lemma:jscheme}. Then $QH^0(X_\Sigma,
\omega)$ is semisimple if and only if the scheme $Z_{W_{F,\Sigma}}$ is
reduced by Corollary \ref{cor:jschemeGen} and Proposition
\ref{prop:LGmodel}.

Recall that $\CO_{X^*}(-K_{X^*})$ is ample, furthermore it is easy
to see that for any $p\in T_M\subset X^*$ the differentials of the
global sections of $\CO_{X^*}(-K_{X^*})$ generate the cotangent
space at $p$. Thus for a general choice of $W\in H^0(X^*,
\CO_{X^*}(-K_{X^*}))$ the critical points of $W$ are
non-degenerate, hence $Z_W$ is reduced by Lemma
\ref{lemma:rednondeg}. Moreover the same is true for a general
section $W\in \Span\{x^{n_\rho}\}$ since log-derivatives of $x^0$
are zeroes. Thus there exists a non-zero polynomial $P\in
\CCC[B_\rho]_{\rho\in\Sigma^1}$ such that $Z_W$ is reduced for any
$W=\sum b_\rho x^{n_\rho}$ with $P(b_\rho)\ne 0$.

Let now $\omega$ be any toric symplectic form on $X_\Sigma$, and
let $F$ be a corresponding piecewise linear function on
$\Sigma$. Notice that by varying $\omega$ we vary $F(n_\rho)$, and
any simultaneous small variation of $F(n_\rho)$ is realized by a
toric symplectic form. Indeed, the fan $\Sigma$ is simplicial thus
any simultaneous variation of $F(n_\rho)$ is realized by a
piecewise linear function, and since $F$ is strictly convex any
small variation gets rise to a strictly convex function. Thus for
a general variation $\omega'$ of $\omega$ all the monomials of $P$
will have different degrees in $s$, hence $P(s^{F'(n_\rho)})\ne 0$, and we are done.
\end{proof}
By a similar argument one can prove the following lemma:
\begin{lemma}\label{lem:ratreal}
Let $X=X_\Sigma$ be a (smooth) toric Fano variety, and let $X^*$ be the dual toric Fano variety over the field $\KK$. Let $V\subset H^0(X^*,\CO_{X^*}(-K^*))$ be a locally closed subvariety defined over $\CCC$. Assume that $\sum s^{F(n_\rho)}x^{n_\rho}\in V$ for some strictly convex piecewise linear function $F$ on the fan $\Sigma$. Then there exists a rational strictly convex piecewise linear function $F'$ on the fan $\Sigma$ such that $\sum s^{F'(n_\rho)}x^{n_\rho}\in V$.
\end{lemma}
\begin{proof}
The variety $\overline{V}$ is given by a system of polynomial equations $P_1(b_\rho)=...=P_k(b_\rho)=0$ for some $P_1,\dotsc ,P_k\in \CCC[B_\rho]_{\rho\in\Sigma^1}$, and $V\subseteq\overline{V}$ is open.

Consider a collection of real numbers $(F_\rho)_{\rho\in\Sigma^1}$. Then $P_i(s^{F_\rho})$ is a formal finite sum of (real) monomials with coefficients in $\CCC$. Assume now that $P_i(s^{F_\rho})=0$. Then there exists a system $L_i$ of {\it linear} equations with {\it integral} coefficients such that $(F_\rho)_{\rho\in\Sigma^1}$ is a solution of $L_i$, and $P_i(s^{F'_\rho})=0$ for {\it any} solution $(F'_\rho)_{\rho\in\Sigma^1}$ of the system $L_i$.

Since $\sum s^{F(n_\rho)}x^{n_\rho}\in V$ there exists a system $L=\cup L_i$ of linear equations with integral coefficients such that $(F(n_\rho))_{\rho\in\Sigma^1}$ is a solution of $L$ and for any solution $(F'_\rho)_{\rho\in\Sigma^1}$ the following holds: $\sum s^{F'_\rho}x^{n_\rho}\in\overline{V}$. Thus there exists a {\it rational} solution of system $L$ obtained from the given one by a small perturbation. Similarly to the proof of Theorem \ref{thm-generic}, any such solution is of the form $(F'(n_\rho))_{\rho\in\Sigma^1}$ where $F'$ is a rational strictly convex piecewise linear function on the fan $\Sigma$. Thus $P_i(s^{F'(n_\rho)})=0$ for all $i$, hence $\sum s^{F'(n_\rho)}x^{n_\rho}\in V$.
\end{proof}

Recall that when $X_\Sigma$ is Fano toric manifold then there exists
a distinguished toric symplectic form $\omega_0$ on $X_\Sigma$ with moment map $\mu_0$,
namely the symplectic form corresponding to $c_1(X_\Sigma)$, i.e. to
the piecewise linear function $F_0$ satisfying $F_0(n_\rho)=-1$ for
all $\rho\in\Sigma^1$. It is the unique symplectic form for which
the corresponding moment polytope is reflexive.

Using the quantum Poincar\'e duality once again,
Theorems~\ref{thmin:sms} and~\ref{thmin:dirsmd} follow from:

\begin{theorem} \label{main-thm} Let $X_\Sigma$ be a smooth toric
Fano manifold, and let $\omega$ be a toric symplectic
form on $X_\Sigma$. Then

(i) If $QH^0(X_\Sigma, \omega_0)$ is semisimple then
$QH^0(X_\Sigma, \omega)$ is semisimple.

(ii) If $QH^0(X_\Sigma, \omega_0)$ contains a field as a direct
summand then so is $QH^0(X_\Sigma, \omega)$.
\end{theorem}
\begin{proof}[{\bf Proof of Theorem~\ref{main-thm}:}]

Let $F$ and $F_0$ be the piecewise linear strictly convex functions corresponding
to $\omega$ and $\omega_0$, and let $W$ and $W_0$ be the
Landau-Ginzburg superpotentials assigned to $F$ and $F_0$. From Proposition~\ref{prop:LGmodel} it follows that
$QH^0(X_\Sigma, \omega)\cong \CO(Z_W)=\KK[N]/J_W$ and $QH^0(X_\Sigma, \omega_0)\cong \CO(Z_{W_0})=\KK[N]/J_{W_0}$, where $W=W_{F,\Sigma}$ and $W_0=W_{F_0,\Sigma}$.
Notice that the loci of sections $W'\in H^0(X_\Sigma^*, \CO(-K_{X_\Sigma^*}))$ for which $Z_{W'}$ is zero dimensional and is not reduced/does not contain a reduced point are locally closed and defined over $\CCC$. Thus, by Lemma \ref{lem:ratreal}, it is sufficient to prove the theorem only for {\it rational} symplectic forms $\omega$.
Furthermore, notice that $s^a\mapsto s^{ak}$ is an automorphism of the field $\KK$,
hence  without loss of generality we may assume that $\omega$ 
is integral. Thus $W,W_0\in\CCC[s^{\pm 1}][N]$.

Next, let $Y=\Spec \CCC[s^{\pm 1}][N]/J_W$ and $Y_0=\Spec
\CCC[s^{\pm 1}][N]/J_{W_0}$, and consider the 
natural projections to $\Spec \CCC[s^{\pm 1}]=\CCC^*$. Notice that
the fibers of $Y$ and $Y_0$ over $s=1$ are canonically isomorphic
since ${W}{\big |}_{s=1} = {W_0}{\big |}_{s=1}$. We denote these
fibers by $Y_c$ (``c" stands for closed). By Lemma
\ref{lemma:jscheme}
$$\dim_\CCC(\CO(Y_c)) = \dim_{\KK} \KK[N] / J_{W_0}
=|\Sigma^d|<\infty,$$ and $Y_0=\Spec \CCC[s^{\pm 1}]\times
Y_c$ since $W_0=s^{-1}\sum_{\rho\in\Sigma^1}x^{n_\rho}$ and $s$ is
invertible in $\CCC[s^{\pm 1}]$.

Consider now the algebras of functions $\CO((Y_0)_\eta)$ and
$\CO(Y_\eta)$ on the generic fibers of
$Y_0\to\Spec\CCC[s^{\pm 1}]$ and $Y\to\Spec\CCC[s^{\pm 1}]$,
i.e.
$$\CO((Y_0)_\eta):=\CO(Y_0)\otimes_{\CCC[s^{\pm
1}]}\CCC(s)\simeq\CO(Y_c)\otimes_\CCC\CCC(s), \  {\rm and} \
\CO(Y_\eta):=\CO(Y)\otimes_{\CCC[s^{\pm 1}]}\CCC(s).$$ Notice that
$\dim_{\CCC(s)}(\CO(Y_\eta))=\dim_\CCC(\CO(Y_c))=|\Sigma^d|<\infty$ by
Lemma \ref{lemma:jscheme}. Notice also that $QH^0(X_\Sigma,
\omega_0)=\CO((Y_0)_\eta)\otimes_{\CCC(s)}\KK$, and
$QH^0(X_\Sigma, \omega)=\CO(Y_\eta)\otimes_{\CCC(s)}\KK$. Thus
$QH^0(X_\Sigma, \omega_0)$ is semisimple over $\KK$ (contains a
field as a direct summand) if and only if $\CO((Y_0)_\eta)$ is
semisimple over $\CCC(s)$ (contains a field as a direct summand)
if and only if $\CO(Y_c)$ is semisimple over $\CCC$ (contains a
field as a direct summand), and $QH^0(X_\Sigma, \omega)$ is
semisimple over $\KK$ (contains a field as a direct summand) if
and only if $\CO(Y_\eta)$ is semisimple over $\CCC(s)$ (contains
a field as a direct summand).

To summarize: all we want to prove is that (i) if $\CO(Y_c)$ is
semisimple over $\CCC$ then $\CO(Y_\eta)$ is semisimple over
$\CCC(s)$, and (ii) if $\CO(Y_c)$ contains a field as a direct
summand then $\CO(Y_\eta)$ contains a field as a direct summand;
or geometrically (i) if $Y_c$ is reduced then $Y_\eta$ is reduced,
and (ii) if $Y_c$ contains a reduced point the $Y_\eta$ contains a
reduced point.

We remark that since we are interested only in the fibers over the
generic point and over $s=1$, we can  replace $\CCC[s^{\pm 1}]$ by
its localization at $s=1$ denoted by $R$. By abuse of notation
$Y\times_{\Spec\CCC[s^{\pm 1}]}\Spec R$ will still be denoted by
$Y$. To complete the proof, we shall need the following two
observations:

\begin{claim}\label{claim:flfin} The map $Y\to\Spec R$ is flat and finite.
\end{claim}

\begin{lemma}\label{lemma:genred} Let $Y$ be flat finite scheme over $\Spec R$,
and let $Y_c$ and $Y_\eta$ be its fibers over the closed and
generic points of $\Spec R$. Then

(i) If $Y_c$ is reduced then $Y_\eta$ is reduced.

(ii) If $Y_c$ contains a reduced point then $Y_\eta$ contains a
reduced point.
\end{lemma}
\noindent The theorem now follows.
\end{proof}

\begin{proof}[\bf Proof of Claim \ref{claim:flfin}.] First let us show flatness.
It is sufficient to check flatness locally on $Y$. Let us fix a closed point $p\in Y\subset\Spec\, R[N]$, hence $p\in Y_c$. We denote $\Spec\, R[N]$ by $T$.
Let $m_1,\dotsc ,m_d$ be a basis of $M=\Hom_\ZZ(N,\ZZ)$ and let
$\partial_i$ be the log derivations defined by $m_i$. Then $J_W$
is generated by $\{\partial_iW\}_{i=1}^d$. We claim that the sequence $\partial_1W,\dotsc ,\partial_dW,s-1$ is a sequence of parameters in the maximal ideal $\fm_p\subset \CO_{T,p}$. Indeed $\dim_pT=d+1$ and $\dim \Spec (\CO_{T,p}/(\partial_1W,\dotsc ,\partial_dW,s-1))=0$, since $\dim_\CCC(\CO_{T,p}/(\partial_1W,\dotsc ,\partial_dW,s-1))=\dim_\CCC \CO_{Y_c,p}\le\dim_\CCC \CO(Y_c)<\infty$. Notice that $T$ is regular, hence Cohen Macaulay, thus $\partial_1W,\dotsc ,\partial_dW,s-1$ is an $\CO_{T,p}-$sequence by \cite{Mat} Theorem 17.4 (iii). Then the local algebra $\CO_{Y,p}=\CO_{T,p}/(\partial_1W,\dotsc ,\partial_dW)$ is Cohen-Macaulay by \cite{Mat} Theorem 17.3 (ii), and $\dim_pY=1$. Flatness at $p$ now follows from \cite{Mat} Theorem 23.1, indeed $Y$ is Cohen-Macaulay at $p$ of dimension $1$, $\Spec R$ is regular of dimension $1$, and the fiber over $s=1$ has dimension $0$.

Remark that flatness of $Y$ over $\Spec R$ implies (and in fact is equivalent to) the following equivalent properties: $s-1$ is not a zero divisor in $\CO(Y)$, and the natural map $\CO(Y)\to\CO(Y_\eta)$ is an embedding. In what follows we shall use these properties many times.

Next we turn to show that $\CO(Y)$  is finite $R$-module. Let
$g_1,\dotsc ,g_l$, $l=|\Sigma^d|$, be a basis of $\CO(Y_c)$ and let $f_1,\dotsc ,f_l$
be its lifting to $\CO(Y)\subset\CO(Y_\eta)$. We claim that
$f_1,\dotsc ,f_l$ freely generate $\CO(Y)$ as
an $R$-module.  Let $\lambda_i(s)\in \CCC(s)$ be elements such
that $0=\sum \lambda_i(s)f_i\in \CO(Y_\eta)$. If not
all $\lambda_i(s)$ are equal to zero, then there exists $k$ such
that $\mu_i(s)=(s-1)^k\lambda_i(s)\in R$ for all $i$ and $\mu_i(1)\ne 0$ for some $i$. Then
$\sum\mu_i(s)f_i=0$ hence $\sum \mu_i(1)g_i=0$ which
is a contradiction. Thus $f_1,\dotsc ,f_l\in
\CO(Y_\eta)$ are linearly independent, and since
$\dim_{\CCC(s)}\CO(Y_\eta)=\dim_\CCC \CO(Y_1)=|\Sigma^d|=l$, they form a
basis of $\CO(Y_\eta)$ over $\CCC(s)$. It remains to show that
$f_1,\dotsc ,f_l$ generate $\CO(Y)$ as
$R$-module. Let $0 \neq f\in \CO(Y)\subset \CO(Y_\eta)$ be
any element then $f=\sum\lambda_i(s)f_i$ for some
$\lambda_i(s)\in \CCC(s)$. As before, if not all the coefficients
$\lambda_i(s)\in R$ then there exists $k>0$ such that
$\mu_i(s)=(s-1)^k\lambda_i(s)\in R$ for all $i$ and $\mu_i(1)\ne
0$ for at least one $i$. Thus $\sum \mu_i(1)g_i\ne 0$ is the class
of $(s-1)^kf$ in $\CO(Y_c)$ which is zero. This is a
contradiction, hence $\CO(Y)$ is a flat finite $R$-module.
\end{proof}

\begin{proof}[\bf Proof of Lemma \ref{lemma:genred}.]
First, notice that the natural map
$\CO(Y)\to \CO(Y_\eta)$ is an embedding since $\CO(Y)$ is flat over $R$. Second, recall that a
flat finite module over a local ring is free, thus $\CO(Y)\simeq
R^l$ as an $R$-module, and $\CO(Y_\eta)\simeq \CCC(s)^l$ as a
$\CCC(s)$-module (vector space); hence for any $0\ne f\in
\CO(Y_\eta)$ there exists a {\it minimal} integer $k$ such that
$(s-1)^kf\in \CO(Y)$.

(i): Assume by contradiction that $Y_\eta$ is not reduced. Then
there exists a nilpotent element $0\ne f\in \CO(Y_\eta)$. Let
$k$ be the minimal integer such that $(s-1)^kf\in \CO(Y)$.
Then $0\ne (s-1)^kf$ is a nilpotent and its class in
$\CO(Y_c)$ is not zero. Thus, we constructed a non-zero nilpotent
in $\CO(Y_c)$, which is a
contradiction.

(ii): Recall that if $Z=\Spec A$, and $A$ is a finite dimensional
algebra over a field then $A=\CO(Z)=\oplus_{q\in Z}\CO_{Z,q}$ as
algebras, where $\CO_{Z,q}$ is the localization of $\CO(Z)$ at $q$
(Chinese remainder theorem). Furthermore any element in the
maximal ideal $\fm_{Z,q}\subset\CO_{Z,q}$ is nilpotent.
Thus $\CO(Y_c)=\oplus_{q\in Y_c}\CO_{Y_c,q}$ as algebras, and
$\CO(Y_\eta)=\oplus_{\epsilon\in Y_\eta}\CO_{Y_\eta,\epsilon}$ as
algebras (hence as $\CO(Y)$-modules).

 Assume that $q\in Y_c$ is a
reduced point. Then $q\in Y$ is a closed point and $\CO_{Y,q}\to \CO_{Y_c,q}=\CCC$  is a
surjective homomorphism from a local ring with kernel generated by
$s-1$, hence $\fm_{Y,q}=(s-1)\CO_{Y,q}$. Tensoring
$\CO(Y_\eta)=\oplus_{\epsilon\in Y_\eta}\CO_{Y_\eta,\epsilon}$
with $\CO_{Y,q}$ over $\CO(Y)$ we obtain the following
decomposition:
$\CO(Y_\eta)\otimes_{\CO(Y)}\CO_{Y,q}=\oplus_{\epsilon\in
Y_\eta}(\CO_{Y_\eta,\epsilon}\otimes_{\CO(Y)}\CO_{Y,q})$. To
finish the proof it is sufficient to show that (a)
$\CO(Y_\eta)\otimes_{\CO(Y)}\CO_{Y,q}$ is a field, and (b)
$\CO_{Y_\eta,\epsilon}\otimes_{\CO(Y)}\CO_{Y,q}$ is either zero or
$\CO_{Y_\eta,\epsilon}$.

For (a), notice that by Nakayma's lemma
$\cap_{k\in\NN}\fm_{Y,q}^k=0$. Thus, any element in $\CO_{Y,q}$ is
of the form $u(s-1)^k$ for some integer $k\ge 0$ and some
invertible element $u\in\CO_{Y,q}$. Next, note that
$s-1\in\CO_{Y,q}$ is not a nilpotent element since otherwise it
would be a zero divisor in $\CO(Y)$, and this contradicts the
flatness of $Y$. Thus $\CO_{Y,q}$ is an integral domain (in fact
it is a DVR) with field of fractions $(\CO_{Y,q})_{s-1}$
(localization of $\CO_{Y,q}$ with respect to $s-1$). Hence
$\CO(Y_\eta)\otimes_{\CO(Y)}\CO_{Y,q}=\CCC(s)\otimes_R\CO_{Y,q}=(\CO_{Y,q})_{s-1}$
is a field.

For (b), let $\fm\subset\CO(Y)\subset\CO(Y_\eta)$ be the maximal
ideal of $q\in Y$ and let $\fn\subset\CO(Y_\eta)$ be the maximal
ideal of $\epsilon$. If $q$ belongs to the closure of $\epsilon$
then $\CO(Y)\setminus\fm\subseteq \CO(Y_\eta)\setminus\fn$, hence
$\CO_{Y_\eta,\epsilon}\otimes_{\CO(Y)}\CO_{Y,q}=\CO_{Y_\eta,\epsilon}$.
If $q$ does not belong to the closure of $\epsilon$ then there
exists $f\in \CO(Y)\setminus\fm$ such that $f\in\fn$. Thus
$1\otimes f\in\CO_{Y_\eta,\epsilon}\otimes_{\CO(Y)}\CO_{Y,q}$ must
be invertible and nilpotent at the same time (any element in
$\fn\CO_{Y_\eta,\epsilon}$ is nilpotent!), hence
$\CO_{Y_\eta,\epsilon}\otimes_{\CO(Y)}\CO_{Y,q}=0$ and we are
done.
\end{proof}

\section{Examples and Counter-Examples} \label{sec:examples}

In this section we prove Proposition~\ref{counter-ex} and
Corollary~\ref{cor:examlpes}. We first provide an example of a
polytope ${\Delta}$ such that the quantum homology subalgebra
$QH_8(X_{\Delta},\omega_0)$ of the corresponding (complex)
$4$-dimensional symplectic toric Fano manifold $X_{\Delta}$ is not
semisimple. Here $\omega_0$ is the distinguished (normalized)
monotone symplectic form on $X_{\Delta}$.

We start by making the identification \begin{equation}
\label{identification} Lie(T)^* = M_{\RR} \simeq \RR^d, \ Lie(T) =
N_{\RR} \simeq (\RR^d)^* \simeq \RR^d \end{equation}
 For technical reasons, it would be easier for us to describe the
vertices of the dual polytope ${\Delta}^*$, that are the
inward-pointing normals to the facets of $\Delta$. Let
$$ {\Delta}^* =  {\rm Conv}
\{e_1,e_2,e_3,e_4,-e_1+e_4,-e_2+e_4,e_2-e_4,-e_2,-e_4,-e_3-e_4
\},$$ where $\{e_1,e_2,e_3,e_4\}$ is the standard basis of
${\mathbb R}^4$. A straightforward computation, whose details we
omit (see remark below), shows that ${\Delta}$ is a Fano Delzant
polytope. We denote by $(X_{ \Delta},\omega_0)$ the corresponding
symplectic toric Fano manifold equipped with the canonical
symplectic form $\omega_0$.

\begin{remark} {\rm
Toric Fano 4-manifolds are completely classified (see
e.g.,~\cite{Bat1},~\cite{Sa}). We refer the reader to the software
package ``PALP "~\cite{KS} with which all the combinatorial data
of the 124 Toric Fano 4-dimensional polytopes can be explicitly
computed. The above example ${ \Delta}$ is the unique reflexive
4-dimensional polytope with 10 vertices, 24 Facets, 11 integer
points, and 59 dual integer points (the ``PALP'' search command
is: ``class.x -di x -He EH:M11V10N59F24L1000"), and it is listed
among the 124 examples in the webpage
    ``http://hep.itp.tuwien.ac.at/$\sim$kreuzer/CY/math/0702890/''.
In Batyrev's classification~\cite{Bat1}, $X_{\Delta}$ appears
under the notation $U_8$ as example number $116$ in section $4$}.
\end{remark}

The corresponding Landau-Ginzburg super potential $W\colon
{\mathbb C}[x_1^{\pm},x_2^{\pm},x_3^{\pm},x_4^{\pm}] \rightarrow
{\mathbb C}$ is given by $$  W = x_1 + x_2 + x_3 + x_4
+ {\frac {1} {x_2}} +{\frac {1} {x_4}} +{\frac {1} {x_3x_4}} +
{\frac {x_4} {x_1}} +{\frac {x_4} {x_2}} + {\frac {x_2} {x_4}}
$$ The partial derivatives are
$$ W_{x_1} = 1- {\frac {x_4} {x_1^2}},  \ W_{x_2} = 1 - {\frac {x_4 + 1 } {x_2^2}} +{\frac 1 {x_4}}, \ W_{x_3} = 1 - {\frac {1} {x_4x_3^2}}, \ W_{x_4} = 1 + {\frac 1 {x_1}} + {\frac 1 {x_2}} - {\frac {x_2 x_3+x_3+1} {x_3x_4^2}}$$
It is easy to check that $x_0 = (-1,-1,-1,1)$ is a critical point of
$W$. On the other hand, the Hessian of W at the point $x_0$ \[ {\rm
Hess}(W_{X_{ \Delta}}(z_0)) = \left( \begin{array}{cccc}
-2 & 0 & 0  & -1\\
0  & -4  & 0 & -2 \\
0 & 0 & -2 & 1 \\
-1 & -2 & 1 & -2 \end{array} \right)\] has rank 3 and hence $x_0$ is
a degenerate critical point. Thus, it follows from
Proposition~\ref{prop:LGmodel}
 and Corollary~\ref{cor:jschemeGen} (i) that $QH^0(X_{\Delta},\omega_0)$
is not semisimple. From the quantum Poincar\'e duality described
in Subsection~\ref{qh-sym-def} we deduce that
$QH_8(X_{\Delta},\omega_0)$ is not semisimple and complete the
proof of Proposition~\ref{counter-ex}.

\begin{remark}
{\rm By taking the product $X_{\Delta} \times \PP^1$ equipped with
the symplectic form $\omega_0 \otimes \alpha \omega_{\PP^1}$, where
$\alpha > > 1$ and $\omega_{\PP^1}$ is the standard symplectic form
on $\PP^1$, we obtain non-monotone symplectic manifolds with non
semisimple quantum homology subalgebra $QH_{10}(X_{\Delta} \times
\PP^1)$}.
\end{remark}

We now turn to sketch of proof of Corollary~\ref{cor:examlpes}.
Note that the combination of McDuff's observation,
Theorem~\ref{thmin:dirsmd}, and Corollary~\ref{cor:jschemeGen} (ii)
reduces the question of the existence of a Calabi quasimorphism
and symplectic quasi-states on a symplectic toric manifold
$(X,\omega)$ to finding a non-degenerate critical point of the
Landau-Ginzburg superpotential corresponding to $(X,\omega_0)$,
where $\omega_0$ is the canonical symplectic form on $X$.

We start with the case of the symplectic blow up of $\PP^d$ at $d+1$ general points. After choosing homogeneous coordinates in an appropriate way we may assume that the $d+1$ points are the zero-dimensional orbits of the natural torus action, hence the blow up admits a structure of a toric variety. The corresponding superpotential
(in the monotone case) is given by
$$ W = \sum_{j=1}^d x_i  + \sum_{j=1}^d {\frac 1 {x_i}} + \prod_{j=1}^d x_i  +  \prod_{j=1}^d {\frac 1 {x_i}}$$
It is easy to check that $(-1,\ldots,-1)$ is a non-degenerate
critical point.

Similarly, for toric Fano 3-folds and 4-folds, one can directly
check (preferably using a computer) that the corresponding
superpotentials have non-degenerate critical points.\footnote{The
combinatorial data required to preform such a computation can be
found e.g. within the database
``http://hep.itp.tuwien.ac.at/~kreuzer/CY/math/0702890/"}

\section{Calabi quasimorphisms} \label{sec:non-unique}

The group-theoretic notion {\it quasimorphism} was originally
introduced with connection to bounded cohomology theory and since
then became an important tool in geometry, topology and dynamics
(see e.g.~\cite{K}). In the context of symplectic geometry, Entov
and Polterovich constructed certain homogeneous quasimorphisms,
called ``Calabi quasimorphism",  and showed several applications to
Hofer's geometry, $C^0$-symplectic topology, and Lagrangian
intersection theory (see e.g.~\cite{EP1},~\cite{EP4}).

We recall that a real-valued function $\Pi$ on a group $G$ is called
a {\it homogeneous quasimorphism} if there is a universal constant
$C>0$ such that for every $g_1, g_2 \in G$:
$$ | \Pi(g_1g_2) - \Pi(g_1)-\Pi(g_2) | \leq C, \ {\rm and} \
\Pi(g^k) = k\Pi(g) \ {\rm for \ every \ } k \in {\mathbb Z} \ {\rm
and} \ g \in G.$$

In~\cite{EP1}, Entov and Polterovich constructed quasimorphisms on
the universal cover ${\widetilde {Ham}} (X)$ of the group of
Hamiltonian diffeomorphisms of a symplectic manifold $X$ using Floer
theory. More precisely, by using spectral invariants which were
defined by Schwarz~\cite{Sc} in the aspherical case, and by
Oh~\cite{Oh1} for general symplectic manifolds (see also
Usher~\cite{Us}). These invariants are given by a map $ c \colon  QH^*(X)
\times {\widetilde {Ham}} (X) \rightarrow \RR$.
We refer the reader
to~\cite{Oh1} and~\cite{MS} for the precise definition of the
spectral invariants and their properties.

Following~\cite{EP1}, for an idempotent $e \in QH^{0}(X,\omega)$ we define
$\mathcal{Q}_e \colon  {\widetilde {Ham}} (X) \rightarrow \RR$ by $
\mathcal{Q}_e = \overline c(e,\cdot)$, where $\overline c(e,\widetilde \phi) = \liminf_{k \rightarrow \infty} {\frac
{c(e,\widetilde \phi^k)} {k}}$ for all $\widetilde\phi \in {\widetilde {Ham}} (X).$
Entov and Polterovich showed that if $e QH^{0}(X,\omega)$ is a field then
$\mathcal{Q}_e$ is a homogenous quasimorphism (see~\cite{EP1} for
the monotone case and~\cite{O},~\cite{EP4} for the general case).
Moreover, $\mathcal{Q}_e$ satisfies the so called {\it Calabi
property}, which means, roughly speaking, that ``locally" it
coincides with the Calabi homomorphism (see~\cite{EP1} for the
precise definition and proof). A natural question raised in~\cite{EP1} asking whether such a
quasimorphism is unique.

Our goal in this section is to prove Corollary~\ref{cor:quasi-non-unique} which
shows that the answer to the question above is negative.
For this we will need some preparation. We start with the
following general property of the spectral invariants
(see~\cite{Oh1},\cite{EP1},\cite{MS}): for every $0 \neq a \in
QH^*(X,\omega)$ and $\gamma \in \pi_1 (Ham(X)) \subset {\widetilde {Ham}}
(X)$ the following holds: $c(a,\gamma) = c(a * {\cal S}(\gamma),\Id) = \log\left\|a * {\cal S}(\gamma)\right\|,$
where ${\cal S}(\gamma) \in  QH^{0}(X,\omega)$ is the Seidel element of $\gamma$ (see e.g.~\cite{MS} for the definition), and
$\|\cdot\|$ is the non-Archimedean norm discussed in Remark \ref{rem:normOnKandQH}. Thus, for every idempotent $e\in QH^0(X,\omega)$ and $\gamma \in \pi_1 (Ham(X))$, we have
\begin{equation} \label{quasi-on-pi1}
 {\cal Q}_{e}(\gamma) = \log\|e * {\cal S}(\gamma)\|_{\rm sp},
\end{equation}
where $\|\cdot\|_{\rm sp}$ is the corresponding non-Archimedean spectral seminorm (cf. subsection \ref{subsec:nonArch}).

Let now $(X_\Sigma,\omega)$ be a symplectic toric Fano manifold, and $F$ be a corresponding strictly convex piecewise linear function on $\Sigma$.
Consider the homomorphisms $\iota\colon N\to \KK[N]/J_W$, $W=W_{F,\Sigma}$, given by the composition
$$ N=\pi_1(T_N) \longrightarrow \pi_1(Ham(X_\Sigma)) \overset{\cal S}\longrightarrow QH^{0}(X_\Sigma,\omega)\simeq\KK[N]/J_W,$$
where ${\cal S}$ is the Seidel map (see e.g~\cite{MS}). By
translating a result of McDuff and Tolman (see Theorem $1.10$ and
Section 5.1 in~\cite{MT} and~\cite{MS} page $441$) to the
Landau-Ginzburg model using~$(\ref{from-Batyrev-to-LG})$, one obtains an explicit formula for $\iota$, namely $\iota(n)=x^n$. To any critical point $p\in Z_W$ one can assign the unit element $e_p\in \CO_{Z_W,p}$, which is an idempotent in $\CO(Z_W)\simeq QH^{0}(X_\Sigma,\omega)$. Furthermore, $e_p\CO(Z_W)=\CO_{Z_W,p}$ is a field if and only if $p$ is a non-degenerate critical point of $W$.
Thus it is sufficient to find two non-degenerate critical points of the superpotential $p,p'\in Z_W$ and $n\in N$ such that
$|x^n(p)|\ne |x^n(p')|$, thanks to Corollary \ref{cor:aboutval} and $(\ref{quasi-on-pi1})$.

Let $(X_\Sigma,F)$ be the blow up of $\PP^2$ at one point equipped with a strictly convex piecewise linear function $F$, or equivalently, with a symplectic form $\omega$ and a moment map $\mu$. After adding a global linear function to $F$ (this operation changes $\mu$, but does not change $\omega$) we may assume that $F(1,0)=0$, $F(0,1)=0$, $F(0,-1)=\beta-\alpha$, and $F(-1,-1)=-\alpha$, where $\alpha>\beta>0$. It is easy to check that $QH^0(X_\Sigma,\omega_0)$ is semisimple, since the superpotential $W_0$ has only non-degenerate critical points. Thus $QH^0(X_\Sigma,\omega)$ is semisimple by Theorem~\ref{thmin:sms}, and $W$ has only non-degenerate critical points.

Recall that the fan $\Sigma$ has four rays generated by $(1,0),(0,1),(0,-1),(-1,-1)$. Set $x_1=x^{(-1,0)}$ and $x_2=x^{(0,1)}$. Then $W=x^{-1}_1+x_2+s^{\beta-\alpha}x_2^{-1}+s^{-\alpha}x_1x_2^{-1},$ and the scheme $Z_W$ of its critical points is given by $-x^{-1}_1+s^{-\alpha} x_1x_2^{-1}=x_2-s^{\beta-\alpha}x_2^{-1}-s^{-\alpha}x_1x_2^{-1}=0,$
or equivalently, $x_1^{4}-s^\alpha x_1-s^{(\beta+\alpha)}=0$ and $x_2=s^{-\alpha} x_1^2$. Assume for simplicity that $\alpha,\beta\in\QQ$. Notice that the Newton diagram of $x_1^{4}-s^\alpha x_1-s^{(\beta+\alpha)}=0$ has two faces if and only if $\alpha<3\beta$; otherwise it has unique face. It is classically known that solutions of such equation correspond to the faces of the Newton diagram; each solution can be written as a Puiseux series in $s$ with non-Archimedean valuation $-l$, where $l$ is the slope of the corresponding face; and the number of solutions (counted with multiplicities) corresponding to a given face is equal to the change of $x_1$ along the face. Thus if $\alpha<3\beta$ and $n=(-1,0)$ then there exist non-degenerate critical points $p,p'\in Z_W$ such that $|x^n(p)|=10^{3/\alpha}\ne 10^{1/\beta}=|x^n(p')|$. Notice that $\omega(L)/\omega(E)=\alpha/\beta$. Corollary~\ref{cor:quasi-non-unique} now follows.

\section{The Critical Values of the Superpotential } \label{sec:mult-by-c1}
Let $(X_\Sigma,F)$ be a smooth toric Fano variety equipped with a strictly convex piecewise linear function $F$, or equivalently, with a symplectic form $\omega$ and a moment map $\mu$. Recall that $c_1(X_\Sigma)$ in Batyrev's description of the (quantum) cohomology is given by
$\sum_{\rho\in\Sigma^1}z_\rho$. Thus, using~$(\ref{from-Batyrev-to-LG})$ to identify Batyrev's description with the Landau-Ginzburg model, one obtains the following formula: $c_1(X_\Sigma)=\sum_{\rho\in\Sigma^1}qs^{F(n_\rho)}x^{n_\rho}=qW$, $W=W_{F,\Sigma}$; hence
$$q^{-1}c_1(X_\Sigma)=W\in\KK[N]/J_W=QH^0(X_\Sigma,\omega).$$
Thus the set of critical values of the superpotential $W$ is equal to the set of eigenvalues of multiplication by $q^{-1}c_1(X_\Sigma)$ on $QH^0(X_\Sigma,\omega)$ by Corollary \ref{cor:aboutval};
which proves Corollary~\ref{cor:mult-by-c1}.

\section*{Appendix: Toric varieties.}
Here we shortly summarize the part of the theory of toric
varieties relevant to our paper. We recall the basic definitions
and some fundamental results (without proofs). The detailed
development of the theory can be found in Fulton's
book~\cite{F93} and in Danilov's survey~\cite{D78}.

As before, throughout the appendix $M$ denotes a lattice of rank $d$, $N=\Hom_\ZZ(M,\ZZ)$
denotes the dual lattice, and $M_\RR=M\otimes_\ZZ\RR$ and $N_\RR=N\otimes_\ZZ\RR$ denote the corresponding pair of dual vector spaces.

\subsection*{Definition of toric varieties and orbit decomposition.}
The references for this subsection are \cite{D78} \S 1, 2, 5, and
\cite{F93} sections 1.2-1.4, 2.1, 2.2, 2.4, 3.1.

A subset $\sigma \subset N_\RR$ is called a {\it rational,
polyhedral cone} if $\sigma$ is a positive span of finitely many
vectors $n_i\in N$, i.e. $\sigma=\Span_{\RR_+}\{n_1,\dotsc ,n_k\}$,
$n_i\in N$. It is not hard to check that $\sigma$ is a rational,
polyhedral cone if and only if there exist $m_1,\dotsc ,m_l\in
M\subset \Hom(N_\RR,\RR)$ such that
$\sigma=\cap_{i=1}^lm_i^{-1}(\RR_+)$. A rational, polyhedral cone
$\sigma$ is called strictly convex if it contains no lines, i.e.
$\sigma \cap (- \sigma) = \{0\}$. For a rational, polyhedral cone
$\sigma\subset N_\RR$ we define the dual cone $\check{\sigma}$ to
be $\check{\sigma} = \{ m \in M_\RR \ | (m, n)\ge 0 \,\,\forall
n\in\sigma \}$, which is again rational and polyhedral. A {\it
face} $\tau$ of a rational, polyhedral cone $\sigma\subset N_\RR$
is defined to be the intersection of $\sigma$ with a supporting
hyperplane, i.e. $\tau=\sigma\cap\Ker(m)$ for some $m\in M_\RR$.
It is easy to see that a face of a (strictly convex) rational,
polyhedral cone is again a (strictly convex) rational, polyhedral
cone. Faces of codimension one are called {\it facets}.

For a strictly convex, rational, polyhedral cone $\sigma$ one can
assign the commutative semigroup $M\cap \check{\sigma}$. Notice
that since $\check{\sigma}$ is rational and polyhedral this
semigroup is finitely generated, hence the semigroup algebra
$\FF[M\cap \check{\sigma}]$ is also finitely generated. We define
affine toric variety $X_\sigma$ over $\FF$ to be $X_\sigma=\Spec\,\FF[M\cap
\check{\sigma}]$. If $\tau\subseteq\sigma$ is a face then
$X_\tau\hookrightarrow X_\sigma$ is an open subvariety. In
particular, since $\sigma$ is strictly convex, the affine toric
variety $X_\sigma$ contains the torus
$X_{\{0\}}=\Spec\,\FF[M]=N\otimes_\ZZ\FF^*=T_N$ as a dense open
subset. Furthermore, the action of torus on itself extends to the
action on $X_\sigma$.

A collection $\Sigma$ of strictly convex, rational, polyhedral
cones in $N_\RR$ is called a {\it fan} if the following two
conditions hold:
\begin{enumerate}
\item If $\sigma\in \Sigma$ and $\tau\subseteq\sigma$ is a face
then $\tau\in \Sigma$. \item If $\sigma,\tau\in \Sigma$ then
$\sigma\cap\tau$ is a common face of $\sigma$ and $\tau$.
\end{enumerate}
A fan $\Sigma$ is called {\it complete} if $\cup_{\sigma\in\Sigma}
\, \sigma=N_\RR$. The set of cones of dimension $k$ in $\Sigma$ is
denoted by $\Sigma^k$, and one-dimensional cones in $\Sigma$ are
called {\it rays}. The primitive integral vector along a ray
$\rho$ is denoted by $n_\rho$.

Given a (complete) fan $\Sigma$ one can construct a (complete)
toric variety $X_\Sigma=\cup_{\sigma\in\Sigma}X_\sigma$ by gluing
$X_\sigma$ and $X_\tau$ along $X_{\sigma\cap\tau}$. Recall that
$X_\Sigma$ has only orbifold singularities if and only if all the
cones in $\Sigma$ are simplicial  (in this case it is called
quasi-smooth); and $X_\Sigma$ is smooth if and only if for any
cone $\sigma\in\Sigma$ the set of primitive integral vectors along
the rays of $\sigma$ forms a part of a basis of the lattice $N$.

The torus $T_N$ acts on $X_\Sigma$ and decomposes it into a
disjoint union of orbits. To a cone $\sigma\in\Sigma$ one can
assign an orbit $O_\sigma\subset X_\sigma$, canonically isomorphic
to $\Spec\,\FF[M\cap\sigma^\bot]$. This defines a one-to-one
order reversing correspondence between the cones in $\Sigma$ and
the orbits in $X_\Sigma$. In particular orbits of codimension one
correspond to rays $\rho\in\Sigma$ and we denote their closures by
$D_\rho$. Thus $\{D_\rho\}_{\rho\in\Sigma^1}$ is the set of
$T_N$-equivariant primitive Weil divisors\footnote{Recall that if
$X$ is a singular variety then one must distinguish between Weil
divisors (i.e. formal finite sums of irreducible subvarieties of
codimension one) and Cartier divisors (i.e. global sections of the
sheaf $\CK^*_X/\CO^*_X$, or equivalently, invertible
subsheaves(=line subbundles) of $\CK$, where $\CK$ denotes the
sheaf of rational functions on $X$). There is a natural
homomorphism $Cartier(X)\to Weil(X)$ and the corresponding
homomorphism between the class groups of divisors $Pic(X)\to
Cl(X)$, however these maps in general need not be surjective or
injective, but for smooth varieties these are isomorphisms. For
any toric variety $X$ these maps are injective, since $X$ is
normal. If in addition $X$ is quasi-smooth then at least
$Pic(X)\otimes_\ZZ\QQ\to Cl(X)\otimes_\ZZ\QQ$ is an isomorphism.}
 on the variety
$X_\Sigma$.

\subsection*{Line bundles on toric varieties.}

The references for this subsection are~\cite{D78} \S 6, 5.8,
and~\cite{F93} sections 3.4, and 1.5.

Let $\Sigma$ be a fan in $N_\RR$ and let $X_\Sigma$ be the
corresponding toric variety. Let $\CL$ be an algebraic line bundle
on $X_\Sigma$. By a trivialization of $\CL$ we mean an isomorphism
$\phi \colon  \CL_{|_{T_N}}\to \CO_{T_N}$ considered up-to the natural
action of $\FF^*$. Recall that any algebraic line bundle on a
torus is trivial, hence any algebraic line bundle $\CL$ on
$X_\Sigma$ can be equipped with a trivialization. To a pair
$(\CL,\phi)$ one can assign a piecewise linear integral function
$F$ on the fan $\Sigma$ (i.e. a function $F$ such that
$F_{|_\sigma}$ is linear for any $\sigma\in\Sigma$ and
$F(N)\subset\ZZ$). This defines a bijective homomorphism between
the group (with respect to the tensor product) of pairs
$(\CL,\phi)$ and the additive group of piecewise linear functions
$F$ as above:
$$F\longleftrightarrow \CO(-\sum_{\rho\in\Sigma^1}F(n_\rho)D_\rho).$$
Furthermore, a change of the trivialization corresponds to adding
a global integral linear function to $F$. In the language of
divisors one can rephrase the above correspondence  as follows:
real/rational/integral piecewise linear functions on the fan
$\Sigma$ are in one-to-one correspondence with
$\RR$/$\QQ$/$\ZZ$-Cartier $T_N$-equivariant divisors. Such
divisors will be called $T$-divisors.

Let $(\CL,\phi)$ be a $T-$divisor, and let $F$ be a corresponding
function. Then $\CL$ is globally generated if and only if $F$ is
convex (i.e. $F(tn+(1-t)n')\ge tF(n)+(1-t)F(n')$ for all $n,n'\in
N_\RR$ and $0\le t\le 1$), and $\CL$ is ample if and only if $F$
is strictly convex (i.e. $F$ is convex, and its maximal linearity
domains are cones in $\Sigma$). Let us now describe the global
sections of $\CL$ in terms of $F$. Any section is completely
determined by its restriction to the big orbit which can be
identified by $\phi$ with an element of $\FF[M]$. Under this
identification the set of global sections of $\CL$ is canonically
isomorphic (up-to the action of $\FF^*$) to the vector space $\Span_\FF\{x^m\}_{m\in
M\cap\Delta_F}$ where $\Delta_F=\Delta_{(\CL,\phi)}=\{m\in
M_\RR\,|\, (m,n_\rho)\ge F(n_\rho) \ {\rm for \ every \ } \rho
\}$. If one changes the trivialization then $\Delta_F$ is
translated by the corresponding element of $M$.

Notice that if $\CL$ is ample then one can reconstruct the fan
$\Sigma$ from the polytope $\Delta_F$. Namely, cones in $\Sigma$
are in one-to-one order reversing correspondence with the faces of
$\Delta_F$. To a face $\gamma\subseteq\Delta_F$ we assign the cone
$\sigma$ being the dual cone to the inner angle of $\Delta_F$ at
$\gamma$ (see \cite{D78} \S 5.8). Furthermore, if $m$ is a vertex
of $\Delta_F$ and $\sigma_m\in\Sigma$ is the corresponding cone,
then $F_{|_{\sigma_m}}=m$. Thus $F$ can also be reconstructed from
the polytope $\Delta_F$.

Recall that the orbits in $X_\Sigma$ are in one-to-one order
reversing correspondence with the cones in $\Sigma$, hence they
are in one-to-one order preserving correspondence with the faces
of $\Delta_F$. Let $\gamma$ be a face of $\Delta_F$, let
$\sigma_\gamma\in\Sigma$ be the corresponding cone, and let
$V=\overline{O}_{\sigma_\gamma}$ be the closure of the
corresponding orbit. Then $V$ has a natural structure of a toric
variety, and the restriction of $\CL$ to $V$ is an ample line
bundle on $V$ defined by the polytope $\gamma-p\subset
\sigma_\gamma^\perp$, where $p\in\gamma$ is any fixed vertex (the
restriction of a trivialized bundle is no longer a trivialized
bundle, this is the reason why one must choose $p$).


\subsection*{Symplectic structure.}

Throughout this subsection the base field is $\FF=\CCC$.
Given an ample $T$-divisor $(\CL,\phi)$ on a toric variety
$X_\Sigma$, one can assign to it a symplectic form
$\omega_{\CL,\phi}$ in the following way: first notice that $\phi$
defines a distinguished (up-to the action of a symmetric group and
up-to a common multiplicative factor) basis in $H^0(X_\Sigma,
\CL^{\otimes r})$ for any $r$. Let
$X_\Sigma\hookrightarrow\PP=\PP(H^0(X_\Sigma, \CL^{\otimes r})^*)$
be the natural embedding (where $r$ is assumed to be large enough).
Recall that projective spaces have canonical symplectic structures
provided by the Fubini-Study forms. Now we simply pull back the
Fubini-Study symplectic form of volume $1$ from $\PP$ to $X_\Sigma$,
and since it is invariant under the action of the symmetric group,
we get a well defined symplectic form on $X_\Sigma$. To make this
construction independent of $r$ and to make the moment polytope compatible with $\Delta_{(\CL,\phi)}$ all we have to do is to multiply the
form by $\frac{2\pi}{r}$. We denote this normalized symplectic form by
$\omega_{\CL,\phi}$ or $\omega_F$ if $F$ is the strictly convex
piecewise linear function associated to $(\CL,\phi)$. Thus
$(\CL,\phi)$ defines the structure of a symplectic toric manifold on
$X_\Sigma$. Furthermore, the action of the compact torus
$T=N\otimes_\ZZ S^1\subset N\otimes_\ZZ\CCC^*=T_N$ is Hamiltonian.
Such a manifold admits a moment map $\mu_{\omega_F} \colon  X\to
Lie(T)^*=M_\RR$. In our case $\mu_{\omega_F}$ is defined by
$$\mu_{\omega_F}(p)=\frac{\sum_{m\in
\Delta_F}|x^m(p)|^2 \, m}{\sum_{m\in \Delta_F}|x^m(p)|^2},$$ and
its image is the polytope $\Delta_F=\Delta_{(\CL,\phi)}$ (cf. \cite{F93} sections 4.1
and 4.2).

\subsection*{Differential Log-forms and the Canonical Class.}

The references for this subsection are \cite{D78} \S 15, and
\cite{F93} sections 4.3.

Let $\Sigma$ be a fan in $N_\RR$ and let $X_\Sigma$ be the
corresponding toric variety. By a log-form we mean a rational
differential 1-form having at worst simple poles along the
components of $X_\Sigma\setminus T_N$. Recall that the sheaf
$\Omega^1_{X_\Sigma}(log)$ of log-forms is trivial vector bundle
canonically isomorphic to $M\otimes_\ZZ\CO_{X_\Sigma}$ (we assign
to $m\in M$ the form $\frac{dx^m}{x^m}$). Moreover there exists an
exact sequence $0\to \Omega^1_{X_\Sigma}\to
\Omega^1_{X_\Sigma}(log)\to \oplus_{\rho\in
\Sigma^1}\CO_{D_\rho}\to 0,$ where the last map is the sum of
residues. It follows from the exact sequence above that
$K_\Sigma=-\sum_{\rho\in\Sigma^1}D_\rho$ is the canonical (Weil)
divisor on $X_\Sigma$. If canonical divisor is $\QQ-$Cartier (e.g.
$X_\Sigma$ is quasi-smooth) then the canonical divisor corresponds
to the rational piecewise linear function $F_K$ defined by the
following property: $F_K(n_\rho)=1$ for any $\rho\in\Sigma^1$.

The dual notion to a log-differential form is a log-derivative.
Log-vector fields also form a trivial vector bundle canonically
isomorphic to $N\otimes_\ZZ\CO_{X_\Sigma}$, namely to any $n\in N$
corresponds the log-derivative $\partial_n$ defined by
$\partial_nx^m=(m,n)x^m$. The notion of log-derivative will be
useful in this paper to make proofs coordinate free.

\bigskip
\noindent Yaron Ostrover\\
Department of Mathematics, M.I.T, Cambridge MA 02139, USA\\
{\it e-mail}: ostrover@math.mit.edu

\bigskip
\noindent Ilya Tyomkin\\
Department of Mathematics, M.I.T, Cambridge MA 02139, USA\\
{\it e-mail}: tyomkin@math.mit.edu

\end{document}